\documentclass[reqno,fleqn]{amsart}
\usepackage{amssymb}
\usepackage{amsmath}
\usepackage{amsfonts}
\usepackage{fancyhdr}

\newtheorem{theorem}{Theorem}
\theoremstyle{plain}

\newtheorem{corollary}{Corollary}

\newtheorem{definition}{Definition}

\newtheorem{lemma}{Lemma}

\newtheorem{proposition}{Proposition}
\newtheorem{remark}{Remark}

\input{tcilatex}

\begin{document}
\title[On a class of nonlinear elliptic, anisotropic singular...]{On a class
of nonlinear elliptic, anisotropic singular perturbations problems}
\author{Ogabi Chokri}
\email{chokri.ogabi@ac-grenoble.fr}
\date{September 30, 2014}
\maketitle
\address{ \ \ \ \ \ \ \ \ \ \ \ \ \ \ \ \ \ \ \ \textit{Academie de
Grenoble, 38000. Grenoble. France}}

\begin{abstract}
In this article we study the asymptotic behavior,as $\epsilon \rightarrow 0$%
, of the solution of a nonlinear elliptic, anisotropic singular
perturbations problem in cylindrical domain, the limit problem is given and
strong convergences are proved, we also give an application to
intergo-differential problems.
\end{abstract}

\section{Description of the problem and main theorems}

The aim of this manuscript is to analyze nonlinear diffusion problems when
the diffusion coefficients in certain directions are going towards zero. We
consider a general nonlinear elliptic singularly perturbed problem which can
be considered as a generalization to some class of integro-differential
problem (see \cite{7}), let us begin by describing the linear part of the
problem as given in \cite{1} and \cite{7}. For $\Omega =\omega _{1}\times
\omega _{2}$ a bounded cylindrical domain of $%
\mathbb{R}
^{N}$ ($N\geq 2$) where $\omega _{1},\omega _{2}$ are Lipschitz domains of $%
\mathbb{R}
^{p}$ and $%
\mathbb{R}
^{N-p}$ respectively, we denote by $x=(x_{1},...,x_{N})=(X_{1},X_{2})$ the
points in $%
\mathbb{R}
^{N}$ where%
\begin{equation*}
X_{1}=(x_{1},...,x_{p})\text{ }\in \omega _{1}\text{and }%
X_{2}=(x_{p+1},...,x_{N})\in \omega _{2},
\end{equation*}

i.e. we split the coordinates into two parts. With this notation we set%
\begin{equation*}
\nabla =(\partial _{x_{1}},...,\partial _{x_{N}})^{T}=\binom{\nabla _{X_{1}}%
}{\nabla _{X_{2}}},
\end{equation*}

where 
\begin{equation*}
\nabla _{X_{1}}=(\partial _{x_{1}},...,\partial _{x_{p}})^{T}\text{ and }%
\nabla _{X_{2}}=(\partial _{x_{p+1}},...,\partial _{x_{N}})^{T}
\end{equation*}

To make it simple we use this abuse of notation 
\begin{equation*}
\nabla _{X_{i}}u\in L^{2}(\Omega )\text{ instead of }\nabla _{X_{i}}u\in %
\left[ L^{2}(\Omega )\right] ^{p;N-p}\text{ for a function }u
\end{equation*}

Let $A=(a_{ij}(x))$ be a $N\times N$ symmetric matrix which satisfies the
ellipticity assumption%
\begin{equation*}
\exists \lambda >0:A\xi \cdot \xi \geq \lambda \left\vert \xi \right\vert
^{2}\text{ }\forall \xi \in R^{N}\text{ for a.e }x\in \Omega ,
\end{equation*}

and 
\begin{equation}
a_{ij}(x)\in L^{\infty }(\Omega ),\forall i,j=1,2,....,N,\text{ }  \label{1}
\end{equation}

where $"\cdot "$ is the canonical scalar product on $R^{N}$. We decompose $A$
into four blocks%
\begin{equation*}
A=\left( 
\begin{array}{cc}
A_{11} & A_{12} \\ 
A_{21} & A_{22}%
\end{array}%
\right) ,
\end{equation*}

where $A_{11}$, $A_{22}$ are respectively $p\times p$ and $(N-p)\times (N-p)$
matrices. For $0<\epsilon \leq 1$ we set%
\begin{equation*}
A_{\epsilon }=\left( 
\begin{array}{cc}
\epsilon ^{2}A_{11} & \epsilon A_{12} \\ 
\epsilon A_{21} & A_{22}%
\end{array}%
\right) ,
\end{equation*}%
\rule{0pt}{0pt}

then we have therefore, for a.e. $x\in \Omega $ and every $\xi \in R^{N}$%
\begin{eqnarray}
A_{\epsilon }\xi \cdot \xi &\geq &\lambda \left( \epsilon ^{2}\left\vert 
\overline{\xi _{1}}\right\vert ^{2}+\left\vert \overline{\xi _{2}}%
\right\vert ^{2}\right) \geq \lambda \epsilon ^{2}\left\vert \xi \right\vert
^{2}\text{ }\forall \xi \in R^{N}\text{ },  \label{3} \\
\text{and }A_{22}\overline{\xi _{2}}\cdot \overline{\xi _{2}} &\geq &\lambda
\left\vert \overline{\xi _{2}}\right\vert ^{2},  \notag
\end{eqnarray}

where we have set%
\begin{equation*}
\xi =\binom{\overline{\xi _{1}}}{\overline{\xi _{2}}},
\end{equation*}

with,%
\begin{equation*}
\overline{\xi _{1}}=(\xi _{1},.....,\xi _{p})^{T}\text{ and }\overline{\xi
_{2}}=(\xi _{p+1},.....,\xi _{N})^{T}
\end{equation*}

And finally let $B:L^{2}(\Omega )\rightarrow L^{2}(\Omega )$ be a nonlinear
locally-Liptchitz operator i.e, for every bounded set $E\subset L^{2}(\Omega
)$ there exists $K_{E}\geq 0$ such that 
\begin{equation}
\forall u,v\in E:\left\Vert B(u)-B(v)\right\Vert _{L^{2}(\Omega )}\leq
K_{E}\left\Vert u-v\right\Vert _{L^{2}(\Omega )},  \label{2}
\end{equation}

,and $B$ satisfies the growth condition%
\begin{equation}
\exists r>2,\text{ }M\geq 0,\text{ }\forall u\in L^{2}(\Omega ):\left\Vert
B(u)\right\Vert _{L^{r}(\Omega )}\leq M\left( 1+\left\Vert u\right\Vert
_{L^{2}(\Omega )}\right) \text{, }  \label{4}
\end{equation}

We define the space%
\begin{equation*}
V=\left\{ u\in L^{2}(\Omega ):\nabla _{X_{1}}u\in L^{2}(\Omega )\right\}
\end{equation*}

Moreover we suppose that for every $E\subset V$ bounded in $L^{2}(\Omega )$
we have%
\begin{equation}
\overline{conv}\left\{ B\left( E\right) \right\} \subset V,  \label{5}
\end{equation}

where $\overline{conv}\left\{ B\left( E\right) \right\} $ is the closed
convex hull of $B\left( E\right) $ in $L^{2}(\Omega ).$This last condition
is the most crucial, it will be used in the proof of the interior estimates
and the convergence theorem.

For $\beta >M\left\vert \Omega \right\vert ^{\frac{1}{2}-\frac{1}{r}}$ we
consider the problem 
\begin{equation}
\left\{ 
\begin{array}{cc}
\dint\limits_{\Omega }A_{\epsilon }\nabla u_{\epsilon }.\nabla \varphi
dx+\beta \dint\limits_{\Omega }u_{\epsilon }\varphi dx\text{ } & 
=\dint\limits_{\Omega }B(u_{\epsilon })\varphi dx\text{, \ }\forall \varphi
\in \mathcal{D}(\Omega ) \\ 
u_{\epsilon }\in H_{0}^{1}(\Omega )\text{ \ \ \ \ \ \ \ \ \ \ \ \ \ \ \ \ }
& 
\end{array}%
\right.  \label{6}
\end{equation}

The existence of $u_{\epsilon }$ will be proved in the next section, Now,
passing to the limit $\epsilon \rightarrow 0$ formally in (\ref{6}) we
obtain the limit problem 
\begin{equation}
\dint\limits_{\Omega }A_{22}\nabla _{X_{2}}u_{0}.\nabla _{X_{2}}\varphi
dx+\beta \dint\limits_{\Omega }u_{0}\varphi dx=\dint\limits_{\Omega
}B(u_{0})\varphi dx\text{, \ }\forall \varphi \in \mathcal{D}(\Omega )
\label{10}
\end{equation}

Our goal is to prove that $u_{0}$ exists and it satisfies (\ref{10}), and
give a sense to the formal convergence $u_{\epsilon }\leadsto u_{0}$,
actually we would like to obtain convergence in $L^{2}(\Omega ).$ We refer
to \cite{1} for more details about the linear theory of problem (\ref{6}).
However the nonlinear theory is poorly known, a monotone problem has been
solved in \cite{2} (using monotonicity argument), and also a case where $B$
is represented by an integral operator has been studied in \cite{7} (in the
last section of this paper, we shall give an application to
integro-differential problems). Generally, in singular perturbation problems
for PDEs, a simple analysis of the problem gives only weak convergences, and
often it is difficult to prove strong convergence, the principal hardness is
the passage to the limit in the nonlinear term. In this article we expose a
resolution method based on the use of several approximated problems
involving regularization with compact operators and truncations. Let us give
the main results.

\begin{theorem}
(Existence and $L^{r}$-regularity of solutions) Assume (\ref{1}), (\ref{3}),
(\ref{4}), and that $B$ is continuous on $L^{2}(\Omega )$ ( not necessarily
locally-Lipschitz) then ($\ref{6}$) has at least a solution $u_{\epsilon
}\in H_{0}^{1}(\Omega ).$ Moreover, if $u_{\epsilon }\in H_{0}^{1}(\Omega )$
is a solution to (\ref{6}) then $\left\Vert u_{\epsilon }\right\Vert
_{L^{r}(\Omega )}\leq \frac{M}{\beta -M\left\vert \Omega \right\vert ^{\frac{%
1}{2}-\frac{1}{r}}}$ for every $\epsilon >0.$
\end{theorem}

For the convergence theorem and the interior estimates we need the following
assumption 
\begin{equation}
\text{ }\partial _{k}A_{22}\text{, }\partial _{i}a_{ij}\text{, }\partial
_{j}a_{ij}\in L^{\infty }(\Omega )\text{ \ }k=1,...,p,\text{\ \ \ }i=1,...,p,%
\text{ \ \ }j=p+1,...,N  \label{34}
\end{equation}

\begin{theorem}
(Interior estimates) Assume (\ref{1}), (\ref{3}), (\ref{2}), (\ref{4}), (\ref%
{5}), (\ref{34})$.$ Let $(u_{\epsilon })\subset H_{0}^{1}(\Omega )$ be a
sequence of solutions to (\ref{6}) then for every open set$\ \Omega ^{\prime
}\subset \subset \Omega $ (i.e $\overline{\Omega ^{\prime }}\subset \Omega $%
) there exists $C_{\Omega ^{\prime }}\geq 0$ (independent of $\epsilon $)
such that%
\begin{equation*}
\forall \epsilon :\left\Vert u_{\epsilon }\right\Vert _{H^{1}(\Omega
^{\prime })}\leq C_{\Omega ^{\prime }}
\end{equation*}
\end{theorem}

\begin{theorem}
( The convergence theorem) Assume (\ref{1}), (\ref{3}), (\ref{2}), (\ref{4}%
), (\ref{5}), (\ref{34})$.$ Let $(u_{\epsilon })\subset H_{0}^{1}(\Omega )$
be a sequence of solutions to (\ref{6}) then there exists a subsequence $%
(u_{\epsilon _{k}})$ and $u_{0}\in H_{loc}^{1}(\Omega )\cap L^{2}(\Omega )$
such that : $\nabla _{X_{2}}u_{0}\in L^{2}(\Omega )$ and 
\begin{equation*}
u_{\epsilon _{k}}\rightarrow u_{0},\nabla _{X_{2}}u_{\epsilon
_{k}}\rightarrow \nabla _{X_{2}}u_{0}\text{ in }L^{2}(\Omega )\text{
strongly as }\epsilon _{k}\rightarrow 0
\end{equation*}

and for a.e $X_{1}$ we have $u_{0}(X_{1},.)\in H_{0}^{1}(\omega _{2}),$and 
\begin{eqnarray}
&&\dint\limits_{\omega _{2}}A_{22}\nabla _{X_{2}}u_{0}(X_{1},.)\cdot \nabla
_{X_{2}}\varphi dX_{2}+\beta \dint\limits_{\omega _{2}}u_{0}(X_{1},.)\varphi
dX_{2}  \notag \\
&=&\dint\limits_{\omega _{2}}B(u_{0})(X_{1},.)\varphi dX_{2},\text{ \ }%
\forall \varphi \in \mathcal{D}(\omega _{2})  \label{11}
\end{eqnarray}
\end{theorem}

\begin{corollary}
If problem (\ref{11}) has a unique solution ( in the sense of \textbf{%
theorem\ 3}) then the convergences given in the previous theorem hold for
the whole sequence $(u_{\epsilon }).$
\end{corollary}

\begin{proof}
The proof is direct, let $(u_{\epsilon })$ be a sequence of solutions to (%
\ref{6}) and suppose that $u_{\epsilon }$ does not converge to $u_{0}$ (as $%
\epsilon \rightarrow 0$) then there exists a subsequence $(u_{\epsilon _{k}})
$ and $\delta >0$ such that $\forall \epsilon _{k},$ $\left\Vert u_{\epsilon
_{k}}-u_{0}\right\Vert _{L^{2}(\Omega )}>$ $\delta $ or $\left\Vert \nabla
_{X_{2}}(u_{\epsilon _{k}}-u_{0})\right\Vert _{L^{2}(\Omega )}>$ $\delta $.
By \textbf{theorem 3} one can extract a subsequence of $(u_{\epsilon _{k}})$
which converges to some $u_{1}$ in the sense of \textbf{theorem 3}, assume
that (\ref{11}) has a unique solution then $u_{1}=u_{0}.$and this
contradicts the previous inequalities.
\end{proof}

In the case of non-uniqueness we can reformulate the convergences ,given in
the previous theorem, using $\epsilon -nets$ like in \cite{7}. Let us recall
the definition of $\epsilon -nets$ (\cite{7})

\begin{definition}
Let $(X,d)$ be metric space, $Y,Y^{\prime }$ two subsets of \ $X,$ then we
say that $Y$ is an $\epsilon -net$ of $Y^{\prime },$ if \ for every $x\in
Y^{\prime }$ there exists an $a\in Y$ such that%
\begin{equation*}
d(x,a)<\epsilon
\end{equation*}
\end{definition}

We define the following space introduced in \cite{7} 
\begin{equation*}
W=\left\{ u\in L^{2}(\Omega ):\nabla _{X_{2}}u\in L^{2}(\Omega ),\text{ and
for a.e }X_{1}\text{, }u(X_{1},.)\in H_{0}^{1}(\omega _{2})\text{ }\right\} ,
\end{equation*}

equipped with the Hilbertian norm (see \cite{7})%
\begin{equation*}
\left\Vert u\right\Vert _{W}^{2}=\left\Vert u\right\Vert _{L^{2}(\Omega
)}^{2}+\left\Vert \nabla _{X_{2}}u\right\Vert _{L^{2}(\Omega )}^{2}
\end{equation*}

Now we can give \textbf{Theorem 3 }in the following form

\begin{theorem}
Under assumptions of theorem 3 then $\Xi $ ,the set of solutions of (\ref{11}%
) in $W,$ is non empty and we have $\Xi \cap H_{loc}^{1}(\Omega )\neq
\varnothing $, and moreover for every $\eta >0,$ there exists $\epsilon
_{0}>0$ such that $\Xi $ is an $\eta -net$ of $\Xi _{\epsilon _{0}}$ in $W$
where%
\begin{equation*}
\Xi _{\epsilon _{0}}=\left\{ u_{\epsilon }\text{ solution to (\ref{6}) for }%
0<\epsilon <\epsilon _{0}\right\}
\end{equation*}
\end{theorem}

\begin{proof}
\textbf{Theorem 1} and\textbf{\ 3} ensure that $\Xi \cap H_{loc}^{1}(\Omega
)\neq \varnothing .$ For the $\eta -net$ convergence, let us reasoning by
contradiction, then there exists $\eta >0$ and a sequence $\epsilon
_{k}\rightarrow 0$ such that $\Xi $ is not an $\eta -net$ of $\Xi _{\epsilon
_{k}}$ in $W$ for every $k$ ( remark that $\Xi _{\epsilon _{k}}\neq
\varnothing $ by \textbf{Theorem 1)} in other words there exists a sequence $%
(u_{\epsilon _{k}^{\prime }})$ with $\epsilon _{k}^{\prime }\rightarrow 0$
such that for every $u_{0}\in \Xi $ we have $\left\Vert u_{\epsilon
_{k}^{\prime }}-u_{0}\right\Vert _{W}\geq \eta $, according to theorem 3
there exists a subsequence of $(u_{\epsilon _{k}^{\prime }})$ which
converges to some $u_{0}\in \Xi $ in$W$ and this contradicts the previous
inequality.
\end{proof}

\section{Existence and $L^{r}-regularity$\ for the solutions and weak
convergences}

\subsection{\textbf{Existence and }$L^{r}-regularity$}

In this subsection we prove \textbf{Theorem 1, }we start by the following
result on the $L^{r}$- regularity for the solutions

\begin{proposition}
Assume (\ref{1}), (\ref{3}), (\ref{4} ) then if $u_{\epsilon }\in
H_{0}^{1}(\Omega )$ is a solution to (\ref{6}) then $u_{\epsilon }\in
L^{r}(\Omega )$ and $\left\Vert u_{\epsilon }\right\Vert _{L^{r}(\Omega
)}\leq \frac{M}{\beta -M\left\vert \Omega \right\vert ^{\frac{1}{2}-\frac{1}{%
r}}}$ for every $\epsilon >0$
\end{proposition}

\begin{proof}
We will proceed as in \cite{3}. Let $u_{\epsilon }\in H_{0}^{1}(\Omega )$ be
a solution to (\ref{6}), given $g\in \mathcal{D}(\Omega )$ and let $%
w_{\epsilon }\in H_{0}^{1}(\Omega )$ be the unique solution to the linear
problem 
\begin{equation}
\dint\limits_{\Omega }A_{\epsilon }\nabla w_{\epsilon }\cdot \nabla \varphi
dx+\beta \dint\limits_{\Omega }w_{\epsilon }\varphi dx\text{ }%
=\dint\limits_{\Omega }g\varphi dx\text{, \ }\forall \varphi \in \mathcal{D}%
(\Omega ),  \label{7}
\end{equation}

the existence of $w_{\epsilon }$ follows by the Lax-Milgram theorem (thanks
to assumptions (\ref{1}), (\ref{3})).

Take $u_{\epsilon }$ as a test function and using the symmetry of $%
A_{\epsilon }$ we get 
\begin{eqnarray*}
\dint\limits_{\Omega }u_{\epsilon }gdx &=&\dint\limits_{\Omega }A_{\epsilon
}\nabla w_{\epsilon }\cdot \nabla u_{\epsilon }dx+\beta \dint\limits_{\Omega
}w_{\epsilon }u_{\epsilon }dx \\
&=&\dint\limits_{\Omega }A_{\epsilon }\nabla u_{\epsilon }\cdot \nabla
w_{\epsilon }dx+\beta \dint\limits_{\Omega }w_{\epsilon }u_{\epsilon }dx \\
&=&\dint\limits_{\Omega }B(u_{\epsilon })w_{\epsilon }dx.
\end{eqnarray*}

Given $s$ such that $\frac{1}{r}+\frac{1}{s}=1$, then by (\ref{4}) we obtain 
\begin{equation}
\left\vert \dint\limits_{\Omega }u_{\epsilon }gdx\right\vert \leq
M(1+\left\Vert u_{\epsilon }\right\Vert _{L^{2}(\Omega )})\left\Vert
w_{\epsilon }\right\Vert _{L^{s}(\Omega )}  \label{8}
\end{equation}

Now we have to estimate $\left\Vert w_{\epsilon }\right\Vert _{L^{s}(\Omega
)}.$ Let $\rho \in C^{1}(%
\mathbb{R}
,%
\mathbb{R}
)$, such that $\rho (0)=0$ and $\rho ^{\prime }\geq 0$ and $\rho ^{\prime
}\in L^{\infty }$ then $\rho (w_{\epsilon })\in H_{0}^{1}(\Omega ),$ take $%
\rho (w_{\epsilon })$ as a test function in (\ref{7}) we get%
\begin{equation*}
\dint\limits_{\Omega }\rho ^{\prime }(w_{\epsilon })A_{\epsilon }\nabla
w_{\epsilon }\cdot \nabla w_{\epsilon }dx+\beta \dint\limits_{\Omega
}w_{\epsilon }\rho (w_{\epsilon })dx\text{ }=\dint\limits_{\Omega }g\rho
(w_{\epsilon })dx\text{.}
\end{equation*}

Now, using ellipticity assumption (\ref{3}) we derive%
\begin{equation*}
\lambda \left( \dint\limits_{\Omega }\rho ^{\prime }(w_{\epsilon
})\left\vert \epsilon \nabla _{X_{1}}w_{\epsilon }\right\vert 
{{}^2}%
dx+\dint\limits_{\Omega }\rho ^{\prime }(w_{\epsilon })\left\vert \nabla
_{X_{2}}w_{\epsilon }\right\vert 
{{}^2}%
dx\right) +\beta \dint\limits_{\Omega }w_{\epsilon }\rho (w_{\epsilon })dx%
\text{ }\leq \dint\limits_{\Omega }g\rho (w_{\epsilon })dx
\end{equation*}

Thus 
\begin{equation*}
\beta \dint\limits_{\Omega }w_{\epsilon }\rho (w_{\epsilon })dx\text{ }\leq
\dint\limits_{\Omega }g\rho (w_{\epsilon })dx
\end{equation*}

Assume that $\forall x\in 
\mathbb{R}
:$ $\left\vert \rho (x)\right\vert \leq \left\vert x\right\vert ^{\frac{1}{%
r-1}},$ so that $\left\vert \rho (x)\right\vert ^{r}\leq \left\vert
x\right\vert \left\vert \rho (x)\right\vert =x\rho (x)$ then, we obtain%
\begin{eqnarray*}
\beta \dint\limits_{\Omega }w_{\epsilon }\rho (w_{\epsilon })dx\text{ }
&\leq &\left\Vert g\right\Vert _{L^{s}(\Omega )}\left( \dint\limits_{\Omega
}\left\vert \rho (w_{\epsilon })\right\vert ^{r}dx\right) ^{\frac{1}{r}} \\
&\leq &\left\Vert g\right\Vert _{L^{s}(\Omega )}\left( \dint\limits_{\Omega
}w_{\epsilon }\rho (w_{\epsilon })dx\text{ }\right) ^{\frac{1}{r}},
\end{eqnarray*}

then 
\begin{equation*}
\beta \left( \dint\limits_{\Omega }w_{\epsilon }\rho (w_{\epsilon })\right)
^{\frac{1}{s}}\leq \left\Vert g\right\Vert _{L^{s}(\Omega )}
\end{equation*}

Now, for $\delta >0$ taking $\rho (x)=x(x^{2}+\delta )^{\frac{s-2}{2}}$ we
show easily that $\rho $ satisfies the above assumptions, so we obtain 
\begin{equation*}
\beta \left( \dint\limits_{\Omega }w_{\epsilon }^{2}(w_{\epsilon
}^{2}+\delta )^{\frac{s-2}{2}}\right) ^{\frac{1}{s}}\leq \left\Vert
g\right\Vert _{L^{s}(\Omega )},
\end{equation*}

let $\delta \rightarrow 0$ by Fatou's lemma we get 
\begin{equation*}
\beta \left\Vert w_{\epsilon }\right\Vert _{L^{s}(\Omega )}\leq \left\Vert
g\right\Vert _{L^{s}(\Omega )}
\end{equation*}

Finally by (\ref{8}) we get 
\begin{equation}
\left\vert \dint\limits_{\Omega }u_{\epsilon }gdx\right\vert \leq \frac{%
M(1+\left\Vert u_{\epsilon }\right\Vert _{L^{2}(\Omega )})}{\beta }%
\left\Vert g\right\Vert _{L^{s}(\Omega )}  \notag
\end{equation}

By density we can take $g\in L^{s}(\Omega )$ and therefore by duality we get 
\begin{equation*}
\left\Vert u_{\epsilon }\right\Vert _{L^{r}(\Omega )}\leq \frac{%
M(1+\left\Vert u_{\epsilon }\right\Vert _{L^{2}(\Omega )})}{\beta },
\end{equation*}

hence by Holder's inequality we obtain%
\begin{equation*}
\left\Vert u_{\epsilon }\right\Vert _{L^{r}(\Omega )}\leq \frac{M}{\beta }+%
\frac{M\left\vert \Omega \right\vert ^{\frac{1}{2}-\frac{1}{r}}}{\beta }%
\left\Vert u_{\epsilon }\right\Vert _{L^{r}(\Omega )},
\end{equation*}

then%
\begin{equation*}
\left\Vert u_{\epsilon }\right\Vert _{L^{r}(\Omega )}\leq \frac{M}{\beta
-M\left\vert \Omega \right\vert ^{\frac{1}{2}-\frac{1}{r}}}
\end{equation*}
\end{proof}

Now, it remains to prove the existence of $u_{\epsilon }$\textbf{, }the
proof is based on the Schauder fixed point theorem. Let $v\in L^{2}(\Omega )$
and $v_{\epsilon }\in H_{0}^{1}(\Omega )$ be the unique solution to the
linearized problem 
\begin{equation}
\dint\limits_{\Omega }A_{\epsilon }\nabla v_{\epsilon }\cdot \nabla \varphi
dx+\beta \dint\limits_{\Omega }v_{\epsilon }\varphi dx\text{ }%
=\dint\limits_{\Omega }B(v)\varphi dx\text{, \ }\forall \varphi \in \mathcal{%
D}(\Omega )  \label{9}
\end{equation}

The existence of $v_{\epsilon }$ follows by the Lax-Milgram theorem ( thanks
to assumptions (\ref{1}), (\ref{3})). Let $\Gamma :L^{2}(\Omega )\rightarrow
L^{2}(\Omega )$ be the mapping defined by $\Gamma (v)=v_{\epsilon }$.We
prove that $\Gamma $ is continuous, fix $v\in L^{2}(\Omega )$ and let $%
v_{n}\rightarrow v$ in $L^{2}(\Omega )$, we note $v_{\epsilon }^{n}=\Gamma
(v_{n})$ then we have%
\begin{equation*}
\dint\limits_{\Omega }A_{\epsilon }\nabla (v_{\epsilon }^{n}-v_{\epsilon
})\cdot \nabla \varphi dx+\beta \dint\limits_{\Omega }(v_{\epsilon
}^{n}-v_{\epsilon })\varphi dx\text{ }=\dint\limits_{\Omega }\left(
B(v_{n})-B(v)\right) \varphi dx\text{, \ }\forall \varphi \in \mathcal{D}%
(\Omega )\text{ }
\end{equation*}

Take $(v_{\epsilon }^{n}-v_{\epsilon })$ as a test function, estimating
using ellipticity assumption (\ref{3}) and Holder's inequality we get%
\begin{equation*}
\beta \left\Vert v_{\epsilon }^{n}-v_{\epsilon }\right\Vert _{L^{2}(\Omega
)}\leq \left\Vert B(v_{n})-B(v)\right\Vert _{L^{2}(\Omega )}.
\end{equation*}

Passing to the limit as $n\rightarrow \infty $ and assume that $B$ is
continuous,\ then the continuity of $\Gamma $ follows. Now, we define the set

\begin{equation*}
S=\left\{ v\in H_{0}^{1}(\Omega ):\left\Vert \nabla v\right\Vert
_{L^{2}(\Omega )}\leq \frac{\sqrt{\beta }}{\epsilon \sqrt{2\lambda }}\left( 
\frac{M\left\vert \Omega \right\vert ^{\frac{1}{2}-\frac{1}{r}}}{\beta
-M\left\vert \Omega \right\vert ^{\frac{1}{2}-\frac{1}{r}}}\right) \text{
and }\left\Vert v\right\Vert _{L^{2}(\Omega )}\leq \frac{M\left\vert \Omega
\right\vert ^{\frac{1}{2}-\frac{1}{r}}}{\beta -M\left\vert \Omega
\right\vert ^{\frac{1}{2}-\frac{1}{r}}}\right\}
\end{equation*}%
It is clear that $S$ is a convex bounded set in $H_{0}^{1}(\Omega )$ and it
is closed in $L^{2}(\Omega )$, then $S$ is compact in $L^{2}(\Omega )$
(thanks to the compact Sobolev embedding\ $H_{0}^{1}(\Omega )\hookrightarrow
L^{2}(\Omega )$). Let us check that $S$ is stable by $\Gamma .$ For $v\in S,$
taking $\varphi =v_{\epsilon }$ in (\ref{9}) and estimating using
ellipticity assumption (\ref{3}) and H\"{o}lder's inequality we get 
\begin{equation*}
\lambda \epsilon ^{2}\left\Vert \nabla v_{\epsilon }\right\Vert
_{L^{2}(\Omega )}^{2}+\beta \left\Vert v_{\epsilon }\right\Vert
_{L^{2}(\Omega )}^{2}\leq \left\Vert B(v)\right\Vert _{L^{2}(\Omega
)}\left\Vert v_{\epsilon }\right\Vert _{L^{2}(\Omega )},
\end{equation*}

then by Young's inequality we derive

\begin{equation*}
\lambda \epsilon ^{2}\left\Vert \nabla v_{\epsilon }\right\Vert
_{L^{2}(\Omega )}^{2}+\beta \left\Vert v_{\epsilon }\right\Vert
_{L^{2}(\Omega )}^{2}\leq \frac{1}{2\beta }\left\Vert B(v)\right\Vert
_{L^{2}(\Omega )}^{2}+\frac{\beta }{2}\left\Vert v_{\epsilon }\right\Vert
_{L^{2}(\Omega )}^{2},
\end{equation*}%
and (\ref{4}) gives%
\begin{eqnarray*}
\lambda \epsilon ^{2}\left\Vert \nabla v_{\epsilon }\right\Vert
_{L^{2}(\Omega )}^{2}+\frac{\beta }{2}\left\Vert v_{\epsilon }\right\Vert
_{L^{2}(\Omega )}^{2} &\leq &\frac{\left\vert \Omega \right\vert ^{1-\frac{2%
}{r}}(M+M\left\Vert v\right\Vert _{L^{2}(\Omega )})^{2}}{2\beta } \\
&\leq &\frac{\left\vert \Omega \right\vert ^{1-\frac{2}{r}}\left( M+\frac{%
M^{2}\left\vert \Omega \right\vert ^{\frac{1}{2}-\frac{1}{r}}}{\beta
-M\left\vert \Omega \right\vert ^{\frac{1}{2}-\frac{1}{r}}}\right) ^{2}}{%
2\beta }\leq \frac{\beta }{2}\left( \frac{M\left\vert \Omega \right\vert ^{%
\frac{1}{2}-\frac{1}{r}}}{\beta -M\left\vert \Omega \right\vert ^{\frac{1}{2}%
-\frac{1}{r}}}\right) ^{2},
\end{eqnarray*}

hence 
\begin{equation*}
\left\{ 
\begin{array}{c}
\left\Vert v_{\epsilon }\right\Vert _{L^{2}(\Omega )}\leq \frac{M\left\vert
\Omega \right\vert ^{\frac{1}{2}-\frac{1}{r}}}{\beta -M\left\vert \Omega
\right\vert ^{\frac{1}{2}-\frac{1}{r}}} \\ 
\left\Vert \nabla v_{\epsilon }\right\Vert _{L^{2}(\Omega )}\leq \frac{\sqrt{%
\beta }}{\epsilon \sqrt{2\lambda }}\left( \frac{M\left\vert \Omega
\right\vert ^{\frac{1}{2}-\frac{1}{r}}}{\beta -M\left\vert \Omega
\right\vert ^{\frac{1}{2}-\frac{1}{r}}}\right)%
\end{array}%
\right.
\end{equation*}

And therefore $v_{\epsilon }=\Gamma (v)\in S.$Whence, there exists at least
a fixed point $u_{\epsilon }\in S$ for $\Gamma ,$ in other words $%
u_{\epsilon }$ is a solution to (\ref{6}).

\subsection{\textbf{Weak convergences as }$\protect\epsilon \rightarrow 0$}

Throughout this article we use the notations $\rightharpoonup $ , $%
\rightarrow $ for weak and strong convergences of sequences respectively.
Assume (\ref{1}), (\ref{3}), (\ref{4} ) and let $(u_{\epsilon })$ be a
sequence of solutions to (\ref{6})\textbf{, }$.$ We begin by a simple
analysis of the problem, considering problem (\ref{6}) and taking $\varphi
=u_{\epsilon }\in H_{0}^{1}(\Omega )$, by ellipticity assumption (\ref{3})
we get 
\begin{equation*}
\lambda \left( \dint\limits_{\Omega }\left\vert \epsilon \nabla
_{X_{1}}u_{\epsilon }\right\vert 
{{}^2}%
dx+\dint\limits_{\Omega }\left\vert \nabla _{X_{2}}u_{\epsilon }\right\vert 
{{}^2}%
dx\right) +\beta \dint\limits_{\Omega }u_{\epsilon }^{2}dx\leq
\dint\limits_{\Omega }B(u_{\epsilon })u_{\epsilon }dx\text{ ,\ \ }
\end{equation*}%
and H\"{o}lder's inequality gives 
\begin{equation*}
\lambda \epsilon ^{2}\left\Vert \nabla _{X_{1}}u_{\epsilon }\right\Vert
_{L^{2}(\Omega )}^{2}+\lambda \left\Vert \nabla _{X_{2}}u_{\epsilon
}\right\Vert _{L^{2}(\Omega )}^{2}+\beta \left\Vert u_{\epsilon }\right\Vert
_{L^{2}(\Omega )}^{2}\leq \left\Vert B(u_{\epsilon })\right\Vert
_{L^{2}(\Omega )}\left\Vert u_{\epsilon }\right\Vert _{L^{2}(\Omega )},
\end{equation*}%
and therefore (\ref{4}) and \textbf{Proposition 1 }give\textbf{\ }%
\begin{equation*}
\lambda \epsilon ^{2}\left\Vert \nabla _{X_{1}}u_{\epsilon }\right\Vert
_{L^{2}(\Omega )}^{2}+\lambda \left\Vert \nabla _{X_{2}}u_{\epsilon
}\right\Vert _{L^{2}(\Omega )}^{2}+\beta \left\Vert u_{\epsilon }\right\Vert
_{L^{2}(\Omega )}^{2}\leq \frac{M^{2}\left\vert \Omega \right\vert ^{1-\frac{%
2}{r}}}{\beta -M\left\vert \Omega \right\vert ^{\frac{1}{2}-\frac{1}{r}}}%
\left( 1+\frac{M\left\vert \Omega \right\vert ^{\frac{1}{2}-\frac{1}{r}}}{%
\beta -M\left\vert \Omega \right\vert ^{\frac{1}{2}-\frac{1}{r}}}\right) 
\end{equation*}%
Whence%
\begin{equation}
\left\{ 
\begin{array}{c}
\left\Vert \epsilon \nabla _{X_{1}}u_{\epsilon }\right\Vert _{L^{2}(\Omega
)}\leq \frac{C}{\sqrt{\lambda }} \\ 
\left\Vert \nabla _{X_{2}}u_{\epsilon }\right\Vert _{L^{2}(\Omega )}\leq 
\frac{C}{\sqrt{\lambda }} \\ 
\left\Vert u_{\epsilon }\right\Vert _{L^{2}(\Omega )}\leq \frac{M\left\vert
\Omega \right\vert ^{\frac{1}{2}-\frac{1}{r}}}{\beta -M\left\vert \Omega
\right\vert ^{\frac{1}{2}-\frac{1}{r}}}%
\end{array}%
\right.   \label{26}
\end{equation}%
$\ \ $,where $C^{2}=\frac{M^{2}\left\vert \Omega \right\vert ^{1-\frac{2}{r}}%
}{\beta -M\left\vert \Omega \right\vert ^{\frac{1}{2}-\frac{1}{r}}}\left( 1+%
\frac{M\left\vert \Omega \right\vert ^{\frac{1}{2}-\frac{1}{r}}}{\beta
-M\left\vert \Omega \right\vert ^{\frac{1}{2}-\frac{1}{r}}}\right) $. Remark
that the gradient of $u_{\epsilon }$ is not bounded uniformly in $%
H_{0}^{1}(\Omega )$ so we cannot obtain strong convergence (using Sobolev
embedding for example) in $L^{2}(\Omega )$, however there exists a
subsequence $(u_{\epsilon _{k}})$ and $u_{0}\in L^{2}(\Omega )$ such that: $%
u_{\epsilon _{k}}\rightharpoonup u_{0}$ , $\nabla _{X_{2}}u_{\epsilon
_{k}}\rightharpoonup \nabla _{X_{2}}u_{0}$ and $\epsilon _{k}\nabla
_{X_{1}}u_{\epsilon _{k}}\rightharpoonup 0$ \ weakly in $L^{2}(\Omega )$ (we
used weak compacity in $L^{2}(\Omega )$, and the continuity of the operator
of derivation on $\mathcal{D}^{\prime }(\Omega )$). The function $u_{0}$
constructed before represents a good candidate for solution to the limit
problems (\ref{10}),(\ref{11}).

\begin{corollary}
We have $u_{0}\in L^{r}(\Omega ).$
\end{corollary}

\begin{proof}
Since $(u_{\epsilon _{k}})$ is bounded in $L^{r}(\Omega )$ then one can
extract a subsequence noted always $(u_{\epsilon _{k}})$ which converges
weakly to some $u_{1}\in L^{r}(\Omega )$ and therefore $u_{\epsilon
_{k}}\rightharpoonup u_{1}$ in $\mathcal{D}^{\prime }(\Omega )$, so $%
u_{1}=u_{0}$
\end{proof}

\section{\protect\bigskip\ Interior estimates and $H_{loc}^{1}-regularity$}

For every $g\in V$ consider the linear problem (\ref{7}), then one can prove
the

\begin{theorem}
Assume (\ref{1}), (\ref{3}), (\ref{34}) then for every $\ \Omega ^{\prime
}\subset \subset \Omega $ (i.e $\overline{\Omega ^{\prime }}\subset \Omega $%
) there exists $C_{\Omega ^{\prime },g}\geq 0$ independent of $\epsilon $
such that%
\begin{equation}
\forall \epsilon :\left\Vert v_{\epsilon }\right\Vert _{H^{1}(\Omega
^{\prime })}\leq C_{\Omega ^{\prime },g}  \label{23}
\end{equation}
\end{theorem}

\begin{proof}
The proof is the same as in \cite{1} (see the rate estimations theorem in 
\cite{1}), remark that the additional term $\beta v_{\epsilon }$ is
uniformly bounded in $L^{2}(\Omega ).$
\end{proof}

To obtain interior estimates for the nonlinear problem we use the well known
Banach-Steinhaus's theorem

\begin{theorem}
\bigskip (see \cite{5}) Let $Y$ and $Z$ be two separated topological vector
spaces, and let $(\mathcal{A}_{\epsilon })$ be a family of continuous linear
mappings from $Y\rightarrow Z$ , $G$ is convex compact set in $Y$. Suppose
that for each $x\in G$ the orbit $\left\{ \mathcal{A}_{\epsilon }(x)\right\}
_{\epsilon }$ is bounded in $Z,$ then $(\mathcal{A}_{\epsilon })$ is
uniformly bounded on $G,$ i.e. there exists a bounded \ $F$ set in $Z$ such
that $\forall \epsilon ,$ $\mathcal{A}_{\epsilon }(G)\subset F.$
\end{theorem}

Now, we are ready to prove \textbf{Theorem 2. }Let $(\Omega _{j})_{j\in 
\mathbb{N}
}$ ,($\forall j:\overline{\Omega _{j}}\subset \Omega _{j+1}$) be an open
covering of $\Omega $, so we can define a family $(p_{j})_{j}$ of seminorms
on $H_{loc}^{1}(\Omega )$ by{}{}\newline
\begin{equation*}
p_{j}(u)=\left\Vert u\right\Vert _{H^{1}(\Omega _{j})}\text{ for every }u\in
H_{loc}^{1}(\Omega )
\end{equation*}

Set $Z=(H_{loc}^{1}(\Omega ),(p_{j})_{j}),$ we can check easily that $Z$ is
a separated locally convex topological vector space where the topology is
generated by the family of seminorms $(p_{j})_{j}$, we also set $%
Y=L^{2}(\Omega )$. We define a family $(\mathcal{A}_{\epsilon })_{\epsilon }$
of linear mappings from $Y$ to $Z$ \ by $\mathcal{A}_{\epsilon
}(g)=v_{\epsilon }$ where $v_{\epsilon }$ is the unique solution to (\ref{7}%
) (existence and uniqueness follows by Lax-Milgram, thanks to (\ref{1}), (%
\ref{3})). $\forall \epsilon $, $\mathcal{A}_{\epsilon }:Y\rightarrow Z$ is
continuous (we can check easily that $\mathcal{A}_{\epsilon }:Y\rightarrow
H^{1}(\Omega )$ and the injection $H^{1}(\Omega )\hookrightarrow Z$ are
continuous). We note $Z_{w}$, $Y_{w}$ the spaces $Z$ and $Y$ equipped with
the weak topology, then for every $\epsilon ,$ $\mathcal{A}_{\epsilon
}:Y_{w}\rightarrow Z_{w}$ is still continuous. Let $E=\left\{ u\in
V:\left\Vert u\right\Vert _{L^{2}(\Omega )}\leq \frac{M\left\vert \Omega
\right\vert ^{\frac{1}{2}-\frac{1}{r}}}{\beta -M\left\vert \Omega
\right\vert ^{\frac{1}{2}-\frac{1}{r}}}\right\} $ and assume (\ref{5}) then $%
G=\overline{conv}\left\{ B\left( E\right) \right\} \subset V$, it is clear
that $G$ is bounded in $Y$ thus $G$ is compact in $Y_{w}.$ Recall that a set
is bounded in a locally convex topological space if and only if the
seminorms that generate the topology are bounded on this set, suppose (\ref%
{34}) then according to (\ref{23}) we have, for $g\in G$, $\left\{ \mathcal{A%
}_{\epsilon }(g)\right\} _{\epsilon }$ is bounded in $Z$, and therefore $%
\left\{ \mathcal{A}_{\epsilon }(g)\right\} _{\epsilon }$ is bounded in $%
Z_{w} $ so by \textbf{Theorem 6} there exists a bounded set $F$ in $Z_{w}$
(also note that $F$ is also bounded in $Z$) such that $\forall \epsilon ,$ $%
\mathcal{A}_{\epsilon }(G)\subset F$. Now let $(u_{\epsilon })$ be a
sequence of solutions to (\ref{6}), and assume in addition (\ref{2}) and (%
\ref{4}) then (\ref{26}) gives\ $(u_{\epsilon })_{\epsilon }\subset E$
whence $(B(u_{\epsilon }))_{\epsilon }\subset G$, and therefore $\mathcal{A}%
_{\epsilon }(B(u_{\epsilon }))\subset F$ for every $\epsilon $, in other
words we have 
\begin{equation*}
\forall j\text{ },\text{ }\exists C_{j}\geq 0\text{ such that }\forall
\epsilon :p_{j}(\mathcal{A}_{\epsilon }(B(u_{\epsilon })))\leq C_{j},
\end{equation*}

where $C_{j}$ is independent of $\epsilon $, and therefore%
\begin{equation*}
\forall \epsilon ,\forall j,\left\Vert u_{\epsilon }\right\Vert
_{H^{1}(\Omega _{j})\text{ }}\leq C_{j}
\end{equation*}

Now, given $\Omega ^{\prime }\subset \subset \Omega $ then there exists $j$
such that $\Omega ^{\prime }\subset \Omega _{j}$ thus%
\begin{equation}
\forall \epsilon ,\left\Vert u_{\epsilon }\right\Vert _{H^{1}(\Omega
^{\prime })\text{ }}\leq C_{j}  \label{24}
\end{equation}

\begin{corollary}
Let $(u_{\epsilon })\subset H_{0}^{1}(\Omega )$ be a sequence of solutions
to (\ref{6}) such that $u_{\epsilon }\rightharpoonup $ $u_{0}$ in $%
L^{2}(\Omega )$ weakly, then under assumptions of \textbf{Theorem 2 }we
have, $u_{0}\in H_{loc}^{1}(\Omega )$
\end{corollary}

\begin{proof}
take $\Omega ^{\prime }\subset \subset \Omega $ an open set, and $\psi \in 
\mathcal{D}(\Omega ^{\prime }),$ $1\leq i\leq N$ then by (\ref{24}) we have%
\begin{equation*}
\left\vert \dint\limits_{\Omega ^{\prime }}u_{\epsilon }\partial _{i}\psi
dx\right\vert =\left\vert \dint\limits_{\Omega ^{\prime }}\partial
_{i}u_{\epsilon }\psi dx\right\vert \leq C_{\Omega ^{\prime }}\left\Vert
\psi \right\Vert _{L%
{{}^2}%
(\Omega ^{\prime })}
\end{equation*}

Let $\epsilon \rightarrow 0$ and using the week convergence $u_{\epsilon
}\rightharpoonup u_{0}$ we get: 
\begin{equation*}
\left\vert \dint\limits_{\Omega ^{\prime }}u_{0}\partial _{i}\psi
dx\right\vert \leq C_{\Omega ^{\prime }}\left\Vert \psi \right\Vert _{L%
{{}^2}%
(\Omega ^{\prime })}\text{ }
\end{equation*}

Hence, $u_{0}\in H_{loc}^{1}(\Omega ).$
\end{proof}

\section{Strong convergence and proof of theorem 3}

Let us begin by some useful propositions

\begin{proposition}
Let $(g_{n})$ be a sequence in $H_{0}^{1}(\Omega )$ and $g\in L^{2}(\Omega )$
such that $\nabla _{X_{2}}g\in L^{2}(\Omega )$ and $\nabla
_{X_{2}}g_{n}\rightarrow \nabla _{X_{2}}g$ in $L^{2}(\Omega )$, then we have:

$g_{n}\rightarrow g$ in $L^{2}(\Omega )$ and for a.e. $X_{1}$ $g(X_{1},.)\in
H_{0}^{1}(\omega _{2})$
\end{proposition}

\begin{proof}
We have for a.e $X_{1}:$ $\nabla _{X_{2}}g_{n}(X_{1},.)\rightarrow \nabla
_{X_{2}}g$ $(X_{1},.)$\ in $L^{2}(\omega _{2})$ (up to a subsequence), and
since for a.e $X_{1}$ and for every $n$ we have $g_{n}(X_{1},.)\in
H_{0}^{1}(\omega _{2})$ then we have for a.e. $X_{1},$ $g(X_{1},.)\in
H_{0}^{1}(\omega _{2}).$And finally the convergence $g_{n}\rightarrow g$ in $%
L^{2}(\Omega )$ follows by Poincar\'{e}'s inequality $\dint\limits_{\Omega
}\left\vert g_{n}-g\right\vert ^{2}\leq C$ $\dint\limits_{\Omega }\left\vert
\nabla _{X_{2}}(g_{n}-g)\right\vert ^{2}$
\end{proof}

\begin{proposition}
\bigskip Let $f,v\in L^{2}(\Omega )$ such that$\ \nabla _{X_{2}}v\in
L^{2}(\Omega )$ and%
\begin{equation*}
\dint\limits_{\Omega }A_{22}\nabla _{X_{2}}v\cdot \nabla _{X_{2}}\varphi
dx+\beta \dint\limits_{\Omega }v\varphi dx=\dint\limits_{\Omega }f\varphi dx%
\text{, \ }\forall \varphi \in \mathcal{D}(\Omega ),
\end{equation*}

then we have for a.e $X_{1}$ 
\begin{equation*}
\dint\limits_{\omega _{2}}A_{22}\nabla _{X_{2}}v(X_{1},.)\cdot \nabla
_{X_{2}}\varphi dX_{2}+\beta \dint\limits_{\omega _{2}}v(X_{1},.)\varphi
dX_{2}=\dint\limits_{\omega _{2}}f(X_{1},.)\varphi dX_{2},\text{ \ }\forall
\varphi \in \mathcal{D}(\omega _{2})\text{\ }
\end{equation*}

Moreover, if for a.e $X_{1}$ we have $v(X_{1},.)\in H_{0}^{1}(\omega _{2})$
then $v$ is the unique function which satisfies the previous equalities
\end{proposition}

\begin{proof}
Same arguments as in \cite{1}.
\end{proof}

\subsection{\textbf{The cut-off problem:}}

Let $\phi \in \mathcal{D}(\Omega )$, and let $(u_{\epsilon })\subset
H_{0}^{1}(\Omega )$ be a sequence of solutions to (\ref{6}) such that $%
u_{\epsilon }$ converges weakly in $L^{2}(\Omega )$ to some $u_{0}\in
L^{2}(\Omega )$. we define $w_{\epsilon }\in H_{0}^{1}(\Omega )$ to be the
unique solution to the cut-off problem (under assumptions (\ref{1}), (\ref{3}%
) existence and uniqueness of $w_{\epsilon }$ follows from the Lax-Milgram
theorem) 
\begin{equation}
\dint\limits_{\Omega }A_{\epsilon }\nabla w_{\epsilon }\cdot \nabla \varphi
dx+\beta \dint\limits_{\Omega }w_{\epsilon }\varphi dx=\dint\limits_{\Omega
}B(\phi u_{\epsilon })\varphi dx\text{, \ }\forall \varphi \in \mathcal{D}%
(\Omega )  \label{12}
\end{equation}

The following Lemma is fundamental in this paper

\begin{lemma}
\bigskip Assume (\ref{1}), (\ref{3}), (\ref{2}),(\ref{4}), (\ref{5}), (\ref%
{34}) then there exists $w_{0}\in $ $W$ such that $w_{\epsilon }\rightarrow $
$w_{0}$ in $W$ strongly and 
\begin{equation*}
\dint\limits_{\Omega }A_{22}\nabla _{X_{2}}w_{0}\cdot \nabla _{X_{2}}\varphi
dx+\beta \dint\limits_{\Omega }w_{0}\varphi dx=\dint\limits_{\Omega }B(\phi
u_{0})\varphi dx\text{, \ }\forall \varphi \in \mathcal{D}(\Omega ),
\end{equation*}%
\begin{multline*}
\dint\limits_{\omega _{2}}A_{22}\nabla _{X_{2}}w_{0}(X_{1},.)\cdot \nabla
_{X_{2}}\varphi dX_{2}+\beta \dint\limits_{\omega _{2}}w_{0}(X_{1},.)\varphi
dX_{2} \\
=\dint\limits_{\omega _{2}}B(\phi u_{0})(X_{1},.)\varphi dX_{2}\text{, \ }%
\forall \varphi \in \mathcal{D}(\omega _{2}),
\end{multline*}

and $w_{0}$ is the unique function which satisfies the two previous weak
formulations.
\end{lemma}

Admit this\textbf{\ }lemma\textbf{\ }for the moment then we have the
following

\begin{proposition}
Assume (\ref{1}), (\ref{3}), (\ref{2}),(\ref{4}), (\ref{5}), (\ref{34}), let 
$(u_{\epsilon })$ be a sequence of solutions to ($\ref{6}$) such that $%
u_{\epsilon }\rightharpoonup u_{0}$ weakly in $L^{2}(\Omega )$, then we have 
$u_{\epsilon }\rightarrow u_{0}$ in $W$ strongly and 
\begin{equation*}
\dint\limits_{\omega _{2}}A_{22}\nabla _{X_{2}}u_{0}(X_{1},.)\cdot \nabla
_{X_{2}}\varphi dX_{2}+\beta \dint\limits_{\omega _{2}}u_{0}(X_{1},.)\varphi
dX_{2}\text{ }=\dint\limits_{\omega _{2}}B(u_{0})(X_{1},.)\varphi dX_{2},%
\text{ \ }\forall \varphi \in \mathcal{D}(\omega _{2})
\end{equation*}
\end{proposition}

\subsection{\protect\bigskip Proof of Proposition 4}

\subsubsection{\textbf{Approximation by truncations}}

Let $(u_{\epsilon })$ be a sequence in $H_{0}^{1}(\Omega )$ of solutions to (%
\ref{6}), assume (\ref{1}), (\ref{3}) and define $w_{\epsilon }^{n}\in
H_{0}^{1}(\Omega )$ the unique solution ( by Lax-Milgram theorem) to the
problem

\begin{equation}
\dint\limits_{\Omega }A_{\epsilon }\nabla w_{\epsilon }^{n}\cdot \nabla
\varphi +\beta \dint\limits_{\Omega }w_{\epsilon }^{n}\varphi
=\dint\limits_{\Omega }B(\phi _{n}u_{\epsilon })\varphi \text{, \ }\forall
\varphi \in \mathcal{D}(\Omega ),  \label{13}
\end{equation}

where $(\phi _{n})$ is a sequence in $\mathcal{D}(\Omega )$ which converges
to $1$ in $L^{\frac{2r}{r-2}}(\Omega ).$

\begin{proposition}
Suppose (\ref{1}), (\ref{3}), (\ref{2}), (\ref{4}) then we have 
\begin{equation*}
\left\Vert \nabla _{X_{2}}w_{\epsilon }^{n}-\nabla _{X_{2}}u_{\epsilon
}\right\Vert _{L^{2}(\Omega )}\rightarrow 0
\end{equation*}%
as $n\rightarrow \infty $ uniformly on $\epsilon $
\end{proposition}

\begin{proof}
\bigskip Subtracting (\ref{6}) from (\ref{13}) and taking $\varphi =$ $%
(w_{\epsilon }^{n}-u_{\epsilon })\in H_{0}^{1}(\Omega )$ we get 
\begin{multline*}
\dint\limits_{\Omega }A_{\epsilon }\nabla (w_{\epsilon }^{n}-u_{\epsilon
})\cdot \nabla (w_{\epsilon }^{n}-u_{\epsilon })dx+\beta
\dint\limits_{\Omega }(w_{\epsilon }^{n}-u_{\epsilon })^{2}dx \\
=\dint\limits_{\Omega }\left( B(\phi _{n}u_{\epsilon })-B(u_{\epsilon
})\right) (w_{\epsilon }^{n}-u_{\epsilon })dx
\end{multline*}%
By (\ref{3}) and H\"{o}lder's inequality we derive%
\begin{equation*}
\lambda \left\Vert \nabla _{X_{2}}(w_{\epsilon }^{n}-u_{\epsilon
})\right\Vert _{L^{2}(\Omega )}^{2}\leq \left\Vert \left( B(\phi
_{n}u_{\epsilon })-B(u_{\epsilon })\right) \right\Vert _{L^{2}}\left\Vert
w_{\epsilon }^{n}-u_{\epsilon }\right\Vert _{L^{2}},
\end{equation*}

and \textbf{Proposition 1 }gives $\left\Vert u_{\epsilon }\right\Vert
_{L^{2}(\Omega )}\leq \frac{M\left\vert \Omega \right\vert ^{\frac{1}{2}-%
\frac{1}{r}}}{\beta -M\left\vert \Omega \right\vert ^{\frac{1}{2}-\frac{1}{r}%
}},$ $\left\Vert \phi _{n}u_{\epsilon }\right\Vert _{L^{2}(\Omega )}\leq 
\frac{M}{\beta -M\left\vert \Omega \right\vert ^{\frac{1}{2}-\frac{1}{r}}}%
\left\Vert \phi _{n}\right\Vert _{L^{\frac{2r}{r-2}}(\Omega )}$, we note $K$
the Lipschitz coefficient of $B$ associated with the bounded set 
\begin{equation*}
\left\{ u\in L^{2}(\Omega ):\left\Vert u\right\Vert _{L^{2}}\leq \underset{n}%
{\sup }(\frac{M\left\vert \Omega \right\vert ^{\frac{1}{2}-\frac{1}{r}}}{%
\beta -M\left\vert \Omega \right\vert ^{\frac{1}{2}-\frac{1}{r}}},\frac{M}{%
\beta -M\left\vert \Omega \right\vert ^{\frac{1}{2}-\frac{1}{r}}}\left\Vert
\phi _{n}\right\Vert _{L^{\frac{2r}{r-2}}})<\infty \right\} ,
\end{equation*}%
$,$

whence (\ref{2}) and H\"{o}lder's inequality give 
\begin{equation*}
\left\Vert \nabla _{X_{2}}(w_{\epsilon }^{n}-u_{\epsilon })\right\Vert
_{L^{2}(\Omega )}^{2}\leq \frac{K}{\lambda }\left\Vert \phi
_{n}-1\right\Vert _{L^{\frac{2r}{r-2}}}\left\Vert u_{\epsilon }\right\Vert
_{L^{r}}\left\Vert w_{\epsilon }^{n}-u_{\epsilon }\right\Vert _{L^{2}}
\end{equation*}

And finally by \textbf{Proposition 1} and Poincar\'{e}'s inequality in the $%
X_{2}$ direction we get 
\begin{equation*}
\left\Vert \nabla _{X_{2}}(w_{\epsilon }^{n}-u_{\epsilon })\right\Vert
_{L^{2}(\Omega )}\leq \frac{C^{\prime }KM}{\lambda (\beta -M\left\vert
\Omega \right\vert ^{\frac{1}{2}-\frac{1}{r}})}\left\Vert \phi
_{n}-1\right\Vert _{L^{\frac{2r}{r-2}}}
\end{equation*}%
\textbf{\ }

Whence $\left\Vert \nabla _{X_{2}}(w_{\epsilon }^{n}-u_{\epsilon
})\right\Vert _{L^{2}(\Omega )}\rightarrow 0$ as $n\rightarrow \infty $
uniformly in $\epsilon $
\end{proof}

\subsubsection{\textbf{The convergence}}

Fix $n$ ,under assumptions of \textbf{Proposition 4} then it follows by 
\textbf{Lemma 1} that there exists $w_{0}^{n}\in W$ such that 
\begin{equation}
w_{\epsilon }^{n}\rightarrow w_{0}^{n}\text{ strongly in }W  \label{22}
\end{equation}

\ and $w_{0}^{n}$ is the unique function in $W$ which satisfies

\begin{equation}
\dint\limits_{\Omega }A_{22}\nabla _{X_{2}}w_{0}^{n}\cdot \nabla
_{X_{2}}\varphi dx+\beta \dint\limits_{\Omega }w_{0}^{n}\varphi
dx=\dint\limits_{\Omega }B(\phi _{n}u_{0})\varphi dx\text{, \ }\forall
\varphi \in \mathcal{D}(\Omega ),  \label{14}
\end{equation}

and for a.e $X_{1}$ we have 
\begin{eqnarray}
&&\dint\limits_{\omega _{2}}A_{22}\nabla _{X_{2}}w_{0}^{n}(X_{1},.)\cdot
\nabla _{X_{2}}\varphi dX_{2}+\beta \dint\limits_{\omega
_{2}}w_{0}^{n}(X_{1},.)\varphi dX_{2}  \label{15} \\
&=&\dint\limits_{\omega _{2}}B(\phi _{n}u_{0})(X_{1},.)\varphi dX_{2},\text{
\ }\forall \varphi \in \mathcal{D}(\omega _{2})  \notag
\end{eqnarray}

For a.e $X_{1}$ taking $\varphi =w_{0}^{n}(X_{1},.)\in H_{0}^{1}(\omega
_{2}) $ in (\ref{15}), by ellipticity assumption (\ref{3}), H\"{o}lder's
inequality we obtain 
\begin{equation*}
\lambda \dint\limits_{\omega _{2}}\left\vert \nabla
_{X_{2}}w_{0}^{n}(X_{1},.)\right\vert ^{2}dX_{2}\leq \left\Vert B(\phi
_{n}u_{0})(X_{1},.)\right\Vert _{L^{2}(\omega _{2})}\left\Vert
w_{0}^{n}(X_{1},.)\right\Vert _{L^{2}(\omega _{2})},
\end{equation*}

and Poincar\'{e}'s inequality in the $X_{2}$ direction gives 
\begin{eqnarray*}
\left\Vert \nabla _{X_{2}}w_{0}^{n}(X_{1},.)\right\Vert _{L^{2}(\omega
_{2})} &\leq &\frac{C^{\prime }}{\lambda }\left\Vert B(\phi
_{n}u_{0})(X_{1},.)\right\Vert _{L^{2}(\omega _{2})} \\
\left\Vert w_{0}^{n}(X_{1},.)\right\Vert _{L^{2}(\omega _{2})} &\leq &\frac{%
C^{\prime 2}}{\lambda }\left\Vert B(\phi _{n}u_{0})(X_{1},.)\right\Vert
_{L^{2}(\omega _{2})}
\end{eqnarray*}

integrating over $\omega _{1}$ yields%
\begin{eqnarray*}
\left\Vert \nabla _{X_{2}}w_{0}^{n}\right\Vert _{L^{2}(\Omega )} &\leq &%
\frac{C^{\prime }}{\lambda }\left\Vert B(\phi _{n}u_{0})\right\Vert
_{L^{2}(\Omega )}, \\
\left\Vert w_{0}^{n}\right\Vert _{L^{2}(\Omega )} &\leq &\frac{C^{\prime 2}}{%
\lambda }\left\Vert B(\phi _{n}u_{0})\right\Vert _{L^{2}(\Omega )},
\end{eqnarray*}

and by (\ref{4}) and H\"{o}lder's inequality (remark that $u_{0}\in
L^{r}(\Omega )$ since $(u_{\epsilon })$ is bounded in $L^{r}(\Omega )$ and $%
u_{\epsilon }\rightharpoonup u_{0}$ in $L^{2}(\Omega )$) we obtain 
\begin{eqnarray*}
\left\Vert \nabla _{X_{2}}w_{0}^{n}\right\Vert _{L^{2}(\Omega )} &\leq &%
\frac{C\left\vert \Omega \right\vert ^{\frac{1}{2}-\frac{1}{r}}M\left(
+\left\Vert \phi _{n}\right\Vert _{L^{\frac{2r}{r-2}}}\left\Vert
u_{0}\right\Vert _{L^{r}}\right) }{\lambda }, \\
\left\Vert w_{0}^{n}\right\Vert _{L^{2}(\Omega )} &\leq &\frac{%
C^{2}\left\vert \Omega \right\vert ^{\frac{1}{2}-\frac{1}{r}}M\left(
+\left\Vert \phi _{n}\right\Vert _{L^{\frac{2r}{r-2}}}\left\Vert
u_{0}\right\Vert _{L^{r}}\right) }{\lambda },
\end{eqnarray*}

(we note that The the right hand sides of the previous inequality is
uniformly bounded). Using weak compacity in $L^{2}(\Omega )$, one can
extract a subsequence noted always $(w_{0}^{n})$ which converges weakly to
some $w_{0}\in L^{2}(\Omega )$ and such that $\nabla
_{X_{2}}w_{0}^{n}\rightharpoonup \nabla _{X_{2}}w_{0}$ weakly. Now, passing
to the limit as $n\rightarrow \infty $ in (\ref{14}) and using 
\begin{equation}
\left\Vert B(\phi _{n}u_{0})-B(u_{0})\right\Vert _{L^{2}(\Omega )}\leq
K\left\Vert \phi _{n}-1\right\Vert _{L^{\frac{2r}{r-2}}}\left\Vert
u_{0}\right\Vert _{L^{r}(\Omega )}  \label{38}
\end{equation}%
we get%
\begin{equation}
\dint\limits_{\Omega }A_{22}\nabla _{X_{2}}w_{0}\cdot \nabla _{X_{2}}\varphi
dx+\beta \dint\limits_{\Omega }w_{0}\varphi dx=\dint\limits_{\Omega
}B(u_{0})\varphi dx\text{, \ }\forall \varphi \in \mathcal{D}(\Omega )
\label{16}
\end{equation}%
Now we will prove that $\nabla _{X_{2}}w_{0}^{n}\rightarrow \nabla
_{X_{2}}w_{0}$ in $L^{2}(\Omega )$ strongly, using ellipticity assumption (%
\ref{3}) we obtain 
\begin{eqnarray}
&&\lambda \left\Vert \nabla _{X_{2}}(w_{0}^{n}-w_{0})\right\Vert
_{L^{2}(\Omega )}^{2}\leq  \label{17} \\
&&\dint\limits_{\Omega }A_{22}\nabla _{X_{2}}(w_{0}^{n}-w_{0})\cdot \nabla
_{X_{2}}(w_{0}^{n}-w_{0})dx+\beta \left\Vert w_{0}^{n}-w_{0}\right\Vert
_{L^{2}(\Omega )}^{2}  \notag \\
&\leq &\dint\limits_{\Omega }A_{22}\nabla _{X_{2}}w_{0}^{n}\cdot \nabla
_{X_{2}}w_{0}^{n}dx-\dint\limits_{\Omega }A_{22}\nabla
_{X_{2}}w_{0}^{n}\cdot \nabla _{X_{2}}w_{0}dx-\dint\limits_{\Omega
}A_{22}\nabla _{X_{2}}w_{0}\cdot \nabla _{X_{2}}w_{0}^{n}dx  \notag \\
&&+\dint\limits_{\Omega }A_{22}\nabla _{X_{2}}w_{0}\cdot \nabla
_{X_{2}}w_{0}dx+\beta \left\Vert w_{0}^{n}-w_{0}\right\Vert _{L^{2}(\Omega
)}^{2}  \notag
\end{eqnarray}%
Taking $\varphi =w_{\epsilon }^{n}$ $\in H_{0}^{1}(\Omega )$ in (\ref{14})
and (\ref{16}) and letting $\epsilon \rightarrow 0$ we get (thanks to (\ref%
{22})) 
\begin{equation}
\dint\limits_{\Omega }A_{22}\nabla _{X_{2}}w_{0}^{n}\cdot \nabla
_{X_{2}}w_{0}^{n}dx+\beta \dint\limits_{\Omega }\left\vert
w_{0}^{n}\right\vert ^{2}dx=\dint\limits_{\Omega }B(\phi
_{n}u_{0})w_{0}^{n}dx,  \label{18}
\end{equation}%
and%
\begin{equation}
\dint\limits_{\Omega }A_{22}\nabla _{X_{2}}w_{0}\cdot \nabla
_{X_{2}}w_{0}^{n}dx+\beta \dint\limits_{\Omega
}w_{0}w_{0}^{n}dx=\dint\limits_{\Omega }B(u_{0})w_{0}^{n}dx  \label{19}
\end{equation}%
Replacing (\ref{18}) and (\ref{19}) in (\ref{17}) we get 
\begin{eqnarray}
&&\lambda \left\Vert \nabla _{X_{2}}(w_{0}^{n}-w_{0})\right\Vert
_{L^{2}(\Omega )}^{2}  \label{20} \\
&\leq &\dint\limits_{\Omega }B(\phi
_{n}u_{0})w_{0}^{n}dx-\dint\limits_{\Omega
}B(u_{0})w_{0}^{n}dx-\dint\limits_{\Omega }A_{22}\nabla
_{X_{2}}w_{0}^{n}\cdot \nabla _{X_{2}}w_{0}dx  \notag \\
&&+\dint\limits_{\Omega }A_{22}\nabla _{X_{2}}w_{0}\cdot \nabla
_{X_{2}}w_{0}dx+\beta \dint\limits_{\Omega }\left\vert w_{0}\right\vert
^{2}dx-\beta \dint\limits_{\Omega }w_{0}w_{0}^{n}dx  \notag
\end{eqnarray}

We have $B(\phi _{n}u_{0})\rightarrow B(u_{0})$ in $L^{2}(\Omega )$ and
since $w_{0}^{n}\rightharpoonup w_{0}$ in $L^{2}(\Omega )$ then%
\begin{equation*}
\dint\limits_{\Omega }B(\phi _{n}u_{0})w_{0}^{n}dx\rightarrow
\dint\limits_{\Omega }B(u_{0})w_{0}dx
\end{equation*}

And since $\nabla _{X_{2}}w_{0}^{n}\rightharpoonup \nabla _{X_{2}}w_{0}$ in $%
L^{2}(\Omega )$ then $A_{22}\nabla _{X_{2}}w_{0}^{n}\rightharpoonup
A_{22}\nabla _{X_{2}}w_{0}$ in $L^{2}(\Omega )$ (since $A_{22}\in L^{\infty
}(\Omega )$). Now, let $n\rightarrow \infty $ in (\ref{20}) we get%
\begin{equation}
\left\Vert \nabla _{X_{2}}(w_{0}^{n}-w_{0})\right\Vert _{L^{2}(\Omega
)}\rightarrow 0  \label{21}
\end{equation}

Thanks to the uniform convergence proved in \textbf{proposition 5}, (\ref{21}%
) and (\ref{22}), we show by the triangular inequality that $\nabla
_{X_{2}}u_{\epsilon }\rightarrow \nabla _{X_{2}}w_{0}$ in $L^{2}(\Omega )$.
Now, we must check that $w_{0}=u_{0}$, according to \textbf{Proposition 2, }%
we have for a.e $X_{1},$\ $w_{0}(X_{1},.)\in H_{0}^{1}(\omega _{2})$ and $%
u_{\epsilon }\rightarrow w_{0}$ in $L^{2}(\Omega )$ , and therefore $%
w_{0}=u_{0}$. By (\ref{16}), we obtain%
\begin{equation*}
\dint\limits_{\Omega }A_{22}\nabla _{X_{2}}u_{0}\cdot \nabla _{X_{2}}\varphi
dx+\beta \dint\limits_{\Omega }u_{0}\varphi dx=\dint\limits_{\Omega
}B(u_{0})\varphi dx\text{, \ }\forall \varphi \in \mathcal{D}(\Omega ),
\end{equation*}

and we finish the proof of \textbf{proposition 4} by using\textbf{\
proposition 3. }Finally, if $(u_{\epsilon })$ is a sequence of solutions to (%
\ref{6}) then there exists a subsequence $(u_{\epsilon _{k}})$ which
converges to some $u_{0}$ in $L^{2}(\Omega )$ weakly ( see subsection 2.2),
whence \textbf{Theorem3 }follows from \textbf{Proposition 4}. Now, it
remains to prove \textbf{Lemma 1} which will be the subject of the next
section.

\section{Proof of Lemma 1}

Before starting , let us give some tools. For $n\in 
\mathbb{N}
^{\ast }$ we note $\Delta _{n}=(I-n^{-1}\Delta )^{-1}$ the resolvent of the
Dirichlet Laplacian on $L^{2}(\Omega )$, this is a compact operator as well
known. Given $f\in L^{2}(\Omega )$ and we note $U_{n}=(I-n^{-1}\Delta )^{-1}f
$ \ , $U_{n}$ is the unique weak solution to the singularly perturbed
problem: 
\begin{equation*}
-\frac{1}{n}\Delta U_{n}+U_{n}=f,
\end{equation*}

we have the

\begin{theorem}
(see \cite{6}): If $f$ $\in H_{0}^{1}(\Omega )$ then : $\left\Vert
U_{n}-f\right\Vert _{L^{2}(\Omega )}\leq C_{\Omega }n^{-\frac{1}{4}%
}\left\Vert f\right\Vert _{H^{1}(\Omega )}$
\end{theorem}

The following lemma will be used in the approximation

\begin{lemma}
For any functions $g\in H_{loc}^{1}(\Omega )\cap L^{2}(\Omega )$ , $\phi \in 
\mathcal{D}(\Omega )$ we have : $\phi g\in H_{0}^{1}(\Omega )$ and moreover
there exists $\Omega ^{\prime }\subset \subset \Omega :$ $\left\Vert \phi
g\right\Vert _{H^{1}(\Omega )}\leq C_{\phi }\left\Vert g\right\Vert
_{H^{1}(\Omega ^{\prime })}$
\end{lemma}

\begin{proof}
the proof is direct.
\end{proof}

\subsection{Approximation of the cut-off problem by \textbf{regularization}}

Let $(u_{\epsilon })\mathbf{\subset }H_{0}^{1}(\Omega )$ be a sequence of
solution to (\ref{6}) such that $u_{\epsilon }\rightharpoonup u_{0}\in
L^{2}(\Omega )$ weakly, assume (\ref{1}), (\ref{3}), (\ref{2}), (\ref{4}), (%
\ref{5}), (\ref{34}). For $\phi \in \mathcal{D}(\Omega )$ fixed we note $%
w_{\epsilon }^{n}\in H_{0}^{1}(\Omega )$ the unique solution to the
following regularized problem (thanks to assumptions (\ref{1}), (\ref{3})
and Lax-Milgram theorem).%
\begin{equation}
\dint\limits_{\Omega }A_{\epsilon }\nabla w_{\epsilon }^{n}\cdot \nabla
\varphi dx+\beta \dint\limits_{\Omega }w_{\epsilon }^{n}\varphi
dx=\dint\limits_{\Omega }B(\Delta _{n}(\phi u_{\epsilon }))\varphi dx,\text{
\ }\forall \varphi \in \mathcal{D}(\Omega )  \label{25}
\end{equation}

\begin{proposition}
As $n\rightarrow \infty $ we have $:$%
\begin{equation*}
\nabla _{X_{2}}w_{\epsilon }^{n}\rightarrow \nabla _{X_{2}}w_{\epsilon }%
\text{ in }L^{2}(\Omega )\text{ uniformly in }\epsilon \text{ ,}
\end{equation*}

where $w_{\epsilon }$ is the solution to the cut-off problem (\ref{12})
\end{proposition}

\begin{proof}
Subtracting (\ref{12}) from (\ref{25}) and taking $\varphi =(w_{\epsilon
}^{n}-w_{\epsilon })\in H_{0}^{1}(\Omega )$ yields%
\begin{multline*}
\dint\limits_{\Omega }A_{\epsilon }\nabla (w_{\epsilon }^{n}-w_{\epsilon
})\cdot \nabla (w_{\epsilon }^{n}-w_{\epsilon })dx+\beta
\dint\limits_{\Omega }(w_{\epsilon }^{n}-w_{\epsilon })^{2}dx \\
=\dint\limits_{\Omega }\left\{ B(\Delta _{n}(\phi u_{\epsilon }))-B((\phi
u_{\epsilon }))\right\} (w_{\epsilon }^{n}-w_{\epsilon })dx
\end{multline*}

Remark that $(\phi u_{\epsilon })_{\epsilon }$ is bounded in $L^{2}(\Omega )$
(\textbf{Proposition 1}) and it is clear that $(\Delta _{n}(\phi u_{\epsilon
}))_{n,\epsilon }$ is bounded in $L^{2}(\Omega )$ ($\left\Vert \Delta
_{n}(\phi u_{\epsilon })\right\Vert _{L^{2}(\Omega )}\leq \left\Vert \phi
u_{\epsilon }\right\Vert _{L^{2}(\Omega )}$), then by ellipticity assumption
(\ref{3}) and the local Lipschitzness of $B$ (\ref{2}) we get 
\begin{multline*}
\lambda \dint\limits_{\Omega }\left\vert \nabla _{X_{2}}(w_{\epsilon
}^{n}-w_{\epsilon })\right\vert ^{2}dx\leq \dint\limits_{\Omega }\left\{
B(\Delta _{n}(\phi u_{\epsilon }))-B(\phi u_{\epsilon })\right\}
(w_{\epsilon }^{n}-w_{\epsilon })dx \\
\leq K^{\prime }\left( \dint\limits_{\Omega }\left\vert \Delta _{n}(\phi
u_{\epsilon })-\phi u_{\epsilon }\right\vert ^{2}dx\right) ^{\frac{1}{2}%
}\left\Vert w_{\epsilon }^{n}-w_{\epsilon }\right\Vert _{L^{2}(\Omega )}
\end{multline*}

Hence, Poincar\'{e}'s inequality gives 
\begin{equation*}
\text{ }\left\Vert \nabla _{X_{2}}(w_{\epsilon }^{n}-w_{\epsilon
})\right\Vert _{L^{2}(\Omega )}\leq \frac{CK^{\prime }}{\lambda }\left(
\dint\limits_{\Omega }\left\vert \Delta _{n}(\phi u_{\epsilon })-\phi
u_{\epsilon }\right\vert ^{2}dx\right) ^{\frac{1}{2}},
\end{equation*}

and by \textbf{Theorem 7} we get 
\begin{equation*}
\left\Vert \nabla _{X_{2}}(w_{\epsilon }^{n}-w_{\epsilon })\right\Vert
_{L^{2}(\Omega )}\leq \frac{C^{\prime }K}{\lambda }n^{-\frac{1}{4}%
}\left\Vert \phi u_{\epsilon }\right\Vert _{H^{1}(\Omega )},
\end{equation*}

and \textbf{Lemma 2 }gives 
\begin{equation*}
\left\Vert \nabla _{X_{2}}(w_{\epsilon }^{n}-w_{\epsilon })\right\Vert
_{L^{2}(\Omega )}\leq \frac{C^{\prime }K}{\lambda }C_{\phi }n^{-\frac{1}{4}%
}\left\Vert u_{\epsilon }\right\Vert _{H^{1}(\Omega ^{\prime })}
\end{equation*}

So finally by \textbf{Theorem 2} we get%
\begin{equation*}
\left\Vert \nabla _{X_{2}}(w_{\epsilon }^{n}-w_{\epsilon })\right\Vert
_{L^{2}(\Omega )}\leq C"n^{-\frac{1}{4}}
\end{equation*}

where $C"\geq 0$ is independent on $\epsilon $ and $n$
\end{proof}

\subsection{\protect\bigskip \textbf{The convergence}}

\subsubsection{\textbf{Passage to the limit as }$\protect\epsilon %
\rightarrow 0$}

Let $n\in 
\mathbb{N}
^{\ast }$ fixed, taking $\varphi =w_{\epsilon }^{n}\in H_{0}^{1}(\Omega )$
in (\ref{25})$,$and estimating using ellipticity assumption (\ref{3}) and (%
\ref{4}) and \textbf{Proposition 1}(as in subsection 2.2) then one can
extract a subsequence $(w_{\epsilon _{k}(n)}^{n})_{k}$ which converges (as $%
\epsilon _{k}(n)\rightarrow 0$) to some $w_{0}^{n}$ in the following
sense\smallskip 
\begin{eqnarray}
w_{\epsilon _{k}(n)}^{n} &\rightharpoonup &w_{0}^{n}\text{ ,}\nabla
_{X_{2}}w_{\epsilon _{k}(n)}^{n}\rightharpoonup \nabla _{X_{2}}w_{0}^{n}%
\text{ and }\epsilon _{k}(n)\nabla _{X_{1}}w_{\epsilon
_{k}(n)}^{n}\rightharpoonup 0  \label{27} \\
&&\text{ in }L^{2}(\Omega )\text{ }  \notag
\end{eqnarray}%
Now passing the limit (as $\epsilon _{k}(n)\rightarrow 0$) in (\ref{25}) we
get\ 
\begin{equation*}
\dint\limits_{\Omega }A_{22}\nabla _{X_{2}}w_{0}^{n}\cdot \nabla
_{X_{2}}\varphi dx+\beta \dint\limits_{\Omega }w_{0}^{n}\varphi dx=\underset{%
\epsilon _{k}(n)\rightarrow 0}{\lim }\dint\limits_{\Omega }B(\Delta
_{n}(\phi u_{\epsilon _{k}(n)}))\varphi dx,\text{ \ }\forall \varphi \in 
\mathcal{D}(\Omega )\smallskip
\end{equation*}%
Since $u_{\epsilon _{k}(n)}\rightharpoonup u_{0}$ weakly in $L^{2}(\Omega )$
then $\phi u_{\epsilon _{k}(n)}\rightharpoonup \phi u_{0}$ weakly in $%
L^{2}(\Omega )$ so by compacity of $\Delta _{n}$ we get $\Delta _{n}(\phi
u_{\epsilon _{k}(n)})\rightarrow \Delta _{n}(\phi u_{0})$ in $L^{2}(\Omega )$
strongly. And therefore, the continuity of $B$ gives $B(\Delta _{n}(\phi
u_{\epsilon _{k}(n)}))\rightarrow B(\Delta _{n}(\phi u_{0}))$ in $%
L^{2}(\Omega )$ strongly, hence the previous equality becomes%
\begin{equation}
\dint\limits_{\Omega }A_{22}\nabla _{X_{2}}w_{0}^{n}\cdot \nabla
_{X_{2}}\varphi dx+\beta \dint\limits_{\Omega }w_{0}^{n}\varphi
dx=\dint\limits_{\Omega }B(\Delta _{n}(\phi u_{0)}))\varphi dx,\text{ \ }%
\forall \varphi \in \mathcal{D}(\Omega )  \label{28}
\end{equation}%
Take $\varphi =w_{\epsilon _{k}(n)}^{n}\in H_{0}^{1}(\Omega )$ in (\ref{28})
and let $\epsilon _{k}(n)\rightarrow 0$ we derive%
\begin{equation}
\dint\limits_{\Omega }A_{22}\nabla _{X_{2}}w_{0}^{n}\cdot \nabla
_{X_{2}}w_{0}^{n}dx+\beta \dint\limits_{\Omega }\left\vert
w_{0}^{n}\right\vert ^{2}dx=\dint\limits_{\Omega }B(\Delta _{n}(\phi
u_{0)}))w_{0}^{n}dx\ \   \label{29}
\end{equation}

Now, we prove strong convergences for the whole sequence (as $\epsilon
\rightarrow 0$)

\begin{proposition}
As $\epsilon \rightarrow 0$ we have $\nabla _{X_{2}}w_{\epsilon
}^{n}\rightarrow \nabla _{X_{2}}w_{0}^{n}$ strongly in $L^{2}(\Omega )$
\end{proposition}

\begin{proof}
Computing%
\begin{multline*}
I_{\epsilon _{k}(n)}^{n}=\dint\limits_{\Omega }A_{\epsilon }\binom{\nabla
_{X_{1}}w_{\epsilon _{k}(n)}^{n}}{\nabla _{X_{2}}(w_{\epsilon
_{k}(n)}^{n}-w_{0}^{n})}\cdot \binom{\nabla _{X_{1}}w_{\epsilon _{k}(n)}^{n}%
}{\nabla _{X_{2}}(w_{\epsilon _{k}(n)}^{n}-w_{0}^{n})}dx+\beta
\dint\limits_{\Omega }\left\vert w_{\epsilon
_{k}(n)}^{n}-w_{0}^{n}\right\vert ^{2}dx \\
=\dint\limits_{\Omega }B(\Delta _{n}(\phi u_{\epsilon _{k}(n)}))\
w_{\epsilon _{k}(n)}^{n}dx-\dint\limits_{\Omega }\epsilon
_{k}(n)A_{12}\nabla _{X_{2}}w_{0}^{n}\cdot \nabla _{X_{1}}w_{\epsilon
_{k}(n)}^{n}dx-2\beta \dint\limits_{\Omega }w_{\epsilon
_{k}(n)}^{n}w_{0}^{n}dx \\
-\dint\limits_{\Omega }\epsilon _{k}(n)A_{21}\nabla _{X_{1}}w_{\epsilon
_{k}(n)}^{n}\cdot \nabla _{X_{2}}w_{0}^{n})dx-\dint\limits_{\Omega
}A_{22}\nabla _{X_{2}}w_{0}^{n}\cdot \nabla _{X_{2}}w_{\epsilon
_{k}(n)}^{n}dx \\
-\dint\limits_{\Omega }A_{22}\nabla _{X_{2}}w_{\epsilon _{k}(n)}^{n}\cdot
\nabla _{X_{2}}w_{0}^{n})dx+\dint\limits_{\Omega }B(\Delta _{n}(\phi
u_{0)}))w_{0}^{n}dx
\end{multline*}

Let $\epsilon _{k}(n)\rightarrow 0$ and using (\ref{27}) we get%
\begin{multline*}
\underset{\epsilon _{k}(n)\rightarrow 0}{\lim }I_{\epsilon _{k}(n)}^{n}=%
\underset{\epsilon _{k}(n)\rightarrow 0}{\lim }\left( \dint\limits_{\Omega
}B(\Delta _{n}(\phi u_{\epsilon _{k}(n)}))\ w_{\epsilon
_{k}(n)}^{n}dx+\dint\limits_{\Omega }B(\Delta _{n}(\phi
u_{0)}))w_{0}^{n}dx\right) -2\beta \dint\limits_{\Omega }\left\vert
w_{0}^{n}\right\vert ^{2}dx \\
-2\dint\limits_{\Omega }A_{22}\nabla _{X_{2}}w_{0}^{n}\cdot \nabla
_{X_{2}}w_{0}^{n}dx
\end{multline*}

Since $B(\Delta _{n}(\phi u_{\epsilon _{k}(n)}))\rightarrow B(\Delta
_{n}(\phi u_{0}))$ in $L^{2}(\Omega )$ strongly and $w_{\epsilon
_{k}(n)}^{n}\rightharpoonup w_{0}^{n}$ weakly then%
\begin{equation*}
\dint\limits_{\Omega }B(\Delta _{n}(\phi u_{\epsilon _{k}(n)}))\ w_{\epsilon
_{k}(n)}^{n}\rightarrow \dint\limits_{\Omega }B(\Delta _{n}(\phi u_{0}))\
w_{0}^{n}
\end{equation*}

Whence by (\ref{29}) we get $\underset{\epsilon _{i}(n)\rightarrow 0}{\lim }%
I_{\epsilon _{i}(n)}^{n}=0$. Now using ellipticity assumption (\ref{3}) we
derive%
\begin{equation*}
\lambda \epsilon _{i}(n)^{2}\dint\limits_{\Omega }\left\vert \nabla
_{X_{1}}w_{\epsilon _{i}(n)}^{n}\right\vert ^{2}+\lambda
\dint\limits_{\Omega }\left\vert \nabla _{X_{2}}(w_{\epsilon
_{i}(n)}^{n}-w_{0}^{n})\right\vert ^{2}\leq I_{\epsilon _{i}(n)}^{n}\text{,}
\end{equation*}

and therefore we get%
\begin{equation*}
\left\Vert \nabla _{X_{2}}(w_{\epsilon _{i}(n)}^{n}-w_{0}^{n}\right\Vert
_{L^{2}(\Omega )}\rightarrow 0\text{ as }\epsilon _{i}(n)\rightarrow 0,
\end{equation*}

According to \textbf{Proposition 2 }we have for a.e $X_{1},$ $%
w_{0}^{n}(X_{1},.)\in H_{0}^{1}(\omega _{2})$ and 
\begin{eqnarray*}
\left\Vert \nabla _{X_{2}}(w_{\epsilon _{i}(n)}^{n}-w_{0}^{n})\right\Vert
_{L^{2}(\Omega )} &\rightarrow &0,\text{ }\left\Vert w_{\epsilon
_{i}(n)}^{n}-w_{0}^{n}\right\Vert _{L^{2}(\Omega )}\rightarrow 0\text{ } \\
\text{as }\epsilon _{i}(n) &\rightarrow &0,
\end{eqnarray*}

By (\ref{28}) and \textbf{Proposition 3 }we show that for every $n$ fixed, $%
w_{0}^{n}$ is the unique function which satisfies for a.e $X_{1}$ 
\begin{multline*}
\dint\limits_{\omega _{2}}A_{22}\nabla _{X_{2}}w_{0}^{n}(X_{1},.)\cdot
\nabla _{X_{2}}\varphi dX_{2}+\beta \dint\limits_{\omega
_{2}}w_{0}^{n}(X_{1},.)\varphi \ dX_{2} \\
=\dint\limits_{\omega _{2}}B(\Delta _{n}(\phi u_{0)}))(X_{1},.)\varphi
dX_{2},\text{ \ }\forall \varphi \in \mathcal{D}(\omega _{2})
\end{multline*}

Since the union of zero measure sets is a zero measure set then we have for
a.e $X_{1\text{ }}$ and $\forall n\in 
\mathbb{N}
^{\ast }$%
\begin{eqnarray}
&&\dint\limits_{\omega _{2}}A_{22}\nabla _{X_{2}}w_{0}^{n}(X_{1},.)\cdot
\nabla _{X_{2}}\varphi dX_{2}+\beta \dint\limits_{\omega
_{2}}w_{0}^{n}(X_{1},.)\varphi dX_{2}\   \label{30} \\
&=&\dint\limits_{\omega _{2}}B(\Delta _{n}(\phi u_{0)}))(X_{1},.)\varphi
dX_{2},\text{ \ }\forall \varphi \in \mathcal{D}(\omega _{2})  \notag
\end{eqnarray}

And finally, the uniqueness of $w_{0}^{n}$ implies that the whole sequence $%
(w_{\epsilon }^{n})$ converges i.e $\forall n\in 
\mathbb{N}
^{\ast }:$%
\begin{equation*}
\left\Vert \nabla _{X_{2}}(w_{\epsilon }^{n}-w_{0}^{n})\right\Vert
_{L^{2}(\Omega )}\rightarrow 0,\text{ and }\left\Vert w_{\epsilon
}^{n}-w_{0}^{n}\right\Vert _{L^{2}(\Omega )}\rightarrow 0\text{ as }\epsilon
\rightarrow 0
\end{equation*}
\end{proof}

\subsubsection{\textbf{Passage to the limit }$n\rightarrow \infty $}

For a.e $X_{1}$ and $\forall n\in 
\mathbb{N}
^{\ast }$ taking $\varphi =w_{0}^{n}(X_{1},.)\in H_{0}^{1}(\omega _{2})$ in (%
\ref{30}), using ellipticity assumption (\ref{3}) and H\"{o}lder's
inequality we get 
\begin{equation*}
\lambda \dint\limits_{\omega _{2}}\left\vert \nabla
_{X_{2}}w_{0}^{n}(X_{1},.)\right\vert ^{2}dX_{2}\leq \left\Vert B(\Delta
_{n}(\phi u_{0}))(X_{1},.)\right\Vert _{L^{2}(\omega _{2})}\left\Vert
w_{0}^{n}(X_{1},.)\right\Vert _{L^{2}(\omega _{2})}
\end{equation*}

and Poincar\'{e}'s inequality in the $X_{2}$ direction gives 
\begin{eqnarray*}
\left\Vert \nabla _{X_{2}}w_{0}^{n}(X_{1},.)\right\Vert _{L^{2}(\omega
_{2})} &\leq &\frac{C^{\prime }}{\lambda }\left\Vert B(\Delta _{n}(\phi
u_{0}))(X_{1},.)\right\Vert _{L^{2}(\omega _{2})} \\
\left\Vert w_{0}^{n}(X_{1},.)\right\Vert _{L^{2}(\omega _{2})} &\leq &\frac{%
C^{\prime 2}}{\lambda }\left\Vert B(\Delta _{n}(\phi
u_{0}))(X_{1},.)\right\Vert _{L^{2}(\omega _{2})}
\end{eqnarray*}

integrating over $\omega _{1}$ yields%
\begin{eqnarray*}
\left\Vert \nabla _{X_{2}}w_{0}^{n}\right\Vert _{L^{2}(\Omega )} &\leq &%
\frac{C^{\prime }}{\lambda }\left\Vert B(\Delta _{n}(\phi u_{0}))\right\Vert
_{L^{2}(\Omega )}, \\
\left\Vert w_{0}^{n}\right\Vert _{L^{2}(\Omega )} &\leq &\frac{C^{\prime 2}}{%
\lambda }\left\Vert B(\Delta _{n}(\phi u_{0}))\right\Vert _{L^{2}(\Omega )},
\end{eqnarray*}

and by (\ref{4}) and Holder's inequality we obtain 
\begin{eqnarray*}
\left\Vert \nabla _{X_{2}}w_{0}^{n}\right\Vert _{L^{2}(\Omega )} &\leq &%
\frac{C^{\prime }\left\vert \Omega \right\vert ^{\frac{1}{2}-\frac{1}{r}%
}M\left( +\left\Vert \phi \right\Vert _{\infty }\left\Vert u_{0}\right\Vert
_{L^{2}}\right) }{\lambda }, \\
\left\Vert w_{0}^{n}\right\Vert _{L^{2}(\Omega )} &\leq &\frac{C^{^{\prime
}2}\left\vert \Omega \right\vert ^{\frac{1}{2}-\frac{1}{r}}M\left(
+\left\Vert \phi \right\Vert _{\infty }\left\Vert u_{0}\right\Vert
_{L^{2}}\right) }{\lambda },
\end{eqnarray*}

(we used the inequality $\left\Vert \Delta _{n}(\phi u_{0})\right\Vert
_{L^{2}(\Omega )}\leq \left\Vert \phi u_{0}\right\Vert _{L^{2}(\Omega )}$
and the notation $\left\Vert \phi \right\Vert _{\infty }=\underset{x\in
\Omega }{\sup }\left\vert \phi (x)\right\vert $)

Whence, it follows by weak compacity that there exists $w_{0}\in
L^{2}(\Omega )$ and a subsequence noted always $(w_{0}^{n})$ such that%
\begin{equation*}
\nabla _{X_{2}}w_{0}^{n}\rightharpoonup \nabla _{X_{2}}w_{0}\text{ and }%
w_{0}^{n}\rightharpoonup w_{0}\text{ in }L^{2}(\Omega )\text{ }
\end{equation*}

Remark that $\phi u_{0}\in H_{0}^{1}(\Omega )$ by \textbf{Lemma2}, then by 
\textbf{Theorem 7 }$\Delta _{n}(\phi u_{0})\rightarrow \phi u_{0}$ in $%
L^{2}(\Omega )$ and therefore, continuity of $B$ gives $B(\Delta _{n}(\phi
u_{0)}))\rightarrow B(\phi u_{0})$ in $L^{2}(\Omega )$. Now, let $%
n\rightarrow \infty $ in (\ref{28}) yields 
\begin{equation}
\dint\limits_{\Omega }A_{22}\nabla _{X_{2}}w_{0}\cdot \nabla _{X_{2}}\varphi
dx+\beta \dint\limits_{\Omega }w_{0}\varphi dx=\dint\limits_{\Omega }B(\phi
u_{0})\varphi dx,\text{ \ }\forall \varphi \in \mathcal{D}(\Omega )
\label{31}
\end{equation}

Take $\varphi =w_{\epsilon }^{n}\in H_{0}^{1}(\Omega )$ in (\ref{31}) and
let $\epsilon \rightarrow 0$ we obtain (by \textbf{Proposition 7)}%
\begin{equation*}
\dint\limits_{\Omega }A_{22}\nabla _{X_{2}}w_{0}\cdot \nabla
_{X_{2}}w_{0}^{n}dx+\beta \dint\limits_{\Omega
}w_{0}w_{0}^{n}dx=\dint\limits_{\Omega }B(\phi u_{0})w_{0}^{n}dx,
\end{equation*}

and as $n\rightarrow \infty $ we derive%
\begin{equation}
\dint\limits_{\Omega }A_{22}\nabla _{X_{2}}w_{0}\cdot \nabla
_{X_{2}}w_{0}dx+\beta \dint\limits_{\Omega }\left\vert w_{0}\right\vert
^{2}dx=\dint\limits_{\Omega }B(\phi u_{0})w_{0}dx  \label{32}
\end{equation}

Now, we prove the strong convergences of $w_{0}^{n}$ and $\nabla
_{X_{2}}w_{0}^{n}$, by ellipticity assumption (\ref{3}), (\ref{29}) and (\ref%
{32}) we get 
\begin{multline*}
\lambda \dint\limits_{\Omega }\left\vert \nabla
_{X_{2}}(w_{0}^{n}-w_{0})\right\vert ^{2}dx+\beta \dint\limits_{\Omega
}\left\vert w_{0}^{n}-w_{0}\right\vert ^{2}dx \\
\leq \dint\limits_{\Omega }A_{22}\nabla _{X_{2}}(w_{0}^{n}-w_{0})\cdot
\nabla _{X_{2}}(w_{0}^{n}-w_{0})dx+\beta \dint\limits_{\Omega }\left\vert
w_{0}^{n}-w_{0}\right\vert ^{2}dx \\
=\dint\limits_{\Omega }B(\Delta _{n}(\phi u_{0)}))w_{0}^{n}dx\
-\dint\limits_{\Omega }A_{22}\nabla _{X_{2}}w_{0}\cdot \nabla
_{X_{2}}w_{0}^{n}dx-2\beta \dint\limits_{\Omega }w_{0}^{n}w_{0}dx \\
-\dint\limits_{\Omega }A_{22}\nabla _{X_{2}}w_{0}^{n}\cdot \nabla
_{X_{2}}w_{0}dx+\dint\limits_{\Omega }B(\phi u_{0})w_{0}dx
\end{multline*}%
Since $B(\Delta _{n}(\phi u_{0)})\rightarrow B(\phi u_{0})$ in $L^{2}(\Omega
)$ and $w_{0}^{n}\rightharpoonup $ $w_{0}$ in $L^{2}(\Omega )$ then%
\begin{equation*}
\dint\limits_{\Omega }B(\Delta _{n}(\phi u_{0)}))w_{0}^{n}\rightarrow
\dint\limits_{\Omega }B(\phi u_{0})w_{0}
\end{equation*}

Let $n\rightarrow \infty $ in the previous inequality we get 
\begin{equation}
\nabla _{X_{2}}w_{0}^{n}\rightarrow \nabla _{X_{2}}w_{0}\text{ in }%
L^{2}(\Omega )  \label{33}
\end{equation}%
Finally by (\ref{33}), \textbf{Proposition 6 }and \textbf{Proposition 7 }and
the triangular inequality we get $\nabla _{X_{2}}w_{\epsilon }\rightarrow
\nabla _{X_{2}}w_{0}$ in $L^{2}(\Omega )$ and therefore (\ref{32}), \textbf{%
Proposition 2 }and\textbf{\ 3 }complete the proof.

\begin{remark}
In addition to convergences given in \textbf{Theorem 3} we also have $%
\epsilon _{k}u_{\epsilon _{k}}\rightarrow 0$ in $L^{2}(\Omega )$ strongly,
indeed ellipticity assumption gives 
\begin{multline*}
\lambda \epsilon _{k}{}^{2}\dint\limits_{\Omega }\left\vert \nabla
_{X_{1}}u_{\epsilon _{k}}\right\vert ^{2}+\lambda \dint\limits_{\Omega
}\left\vert \nabla _{X_{2}}(u_{\epsilon _{k}}-u_{0})\right\vert ^{2} \\
\leq \dint\limits_{\Omega }A_{\epsilon }\binom{\nabla _{X_{1}}u_{\epsilon
_{k}}}{\nabla _{X_{2}}(u_{\epsilon _{k}}-u_{0})}\cdot \binom{\nabla
_{X_{1}}u_{\epsilon _{k}}}{\nabla _{X_{2}}(u_{\epsilon _{k}}-u_{0})}dx+\beta
\dint\limits_{\Omega }\left\vert u_{\epsilon _{k}}-u_{0}\right\vert ^{2}dx,
\end{multline*}
\end{remark}

and we can prove easily that the right-hand side of this inequality
converges to $0.$

\section{Some Applications}

\subsection{\textbf{A regularity result and rate of convergence}}

In this subsection we make some additional assumptions, suppose that for
every $u\in L^{2}(\Omega ),$ 
\begin{equation}
\nabla _{X_{1}}B(u)\in L^{2}(\Omega ),  \label{43}
\end{equation}

and for every $\rho \in \mathcal{D(\omega }_{1}\mathcal{)}$ and $u,v$ $\in
L^{2}(\Omega )$ we have 
\begin{equation}
\left\Vert \rho B(u)-\rho B(v)\right\Vert _{L^{2}}\leq \left\Vert B(\rho
u)-B(\rho v)\right\Vert  \label{44}
\end{equation}

Remark that\textbf{\ Theorem 3} of section 1 gives only $H_{loc}^{1}-$
regularity for $u_{0}$, however we have the following

\begin{proposition}
Under assumptions of \textbf{Theorem 3 }and (\ref{43}) we have $u_{0}\in
H^{1}(\Omega ),$
\end{proposition}

\begin{proof}
We will proceed as in \cite{1}, let $\omega _{1}^{\prime }\subset \subset
\omega _{1}$, for $0<h<d(\omega _{1}^{\prime },\omega _{1}),$ $X_{1}\in
\omega _{1}^{\prime }$ we set $\tau
_{h}^{i}u_{0}(X_{1},X_{2})=u_{0}(X_{1}+he_{i},X_{2})$ $i=1,...,p.$ From (\ref%
{11}) we have%
\begin{multline*}
\dint\limits_{\omega _{2}}\tau _{h}^{i}A_{22}\nabla _{X_{2}}(\tau
_{h}^{i}u_{0}-u_{0})\nabla _{X_{2}}\varphi dX_{2}+\dint\limits_{\omega
_{2}}(\tau _{h}^{i}A_{22}-A_{22})\nabla _{X_{2}}u_{0}\nabla _{X_{2}}\varphi
dX_{2} \\
+\beta \dint\limits_{\omega _{2}}(\tau _{h}^{i}u_{0}-u_{0})\varphi
dX_{2}=\dint\limits_{\omega _{2}}\left\{ \tau
_{h}^{i}B(u_{0})-B(u_{0})\right\} \varphi dX_{2}\text{ }
\end{multline*}

Taking $\varphi =\frac{\tau _{h}^{i}u_{0}-u_{0}}{h^{2}}$ as a test function,
using ellipticity assumption (\ref{3}) and H\"{o}lder's inequality we derive 
\begin{multline*}
\lambda \left\Vert \nabla _{X_{2}}\left( \frac{\tau _{h}^{i}u_{0}-u_{0}}{h}%
\right) \right\Vert _{L^{2}(\omega _{2})}^{2}\leq \\
\left\Vert \left( \frac{\tau _{h}^{i}A_{22}-A_{22}}{h}\right) \nabla
_{X_{2}}u_{0}\right\Vert _{L^{2}(\omega _{2})}\left\Vert \nabla
_{X_{2}}\left( \frac{\tau _{h}^{i}u_{0}-u_{0}}{h}\right) \right\Vert
_{L^{2}(\omega _{2})} \\
+\left\Vert \left( \frac{\tau _{h}^{i}B(u_{0})-B(u_{0})}{h}\right)
\right\Vert _{L^{2}(\omega _{2})}\left\Vert \left( \frac{\tau
_{h}^{i}u_{0}-u_{0}}{h}\right) \right\Vert _{L^{2}(\omega _{2})}
\end{multline*}

Using Poincar\'{e}'s inequality we deduce%
\begin{equation*}
\left\Vert \frac{\tau _{h}^{i}u_{0}-u_{0}}{h}\right\Vert _{L^{2}(\omega
_{2})}\leq \frac{C}{\lambda }\left\{ 
\begin{array}{c}
\left\Vert \frac{\tau _{h}^{i}A_{22}-A_{22}}{h}\right\Vert _{L^{\infty
}(\omega _{2})}\left\Vert \nabla _{X_{2}}u_{0}\right\Vert _{L^{2}(\omega
_{2})} \\ 
+\left\Vert \frac{\tau _{h}^{i}B(u_{0})-B(u_{0})}{h}\right\Vert
_{L^{2}(\omega _{2})}%
\end{array}%
\right\}
\end{equation*}

Using regularity assumption (\ref{34}) and integrating over $\omega
_{1}^{\prime }$ (we use only the assumption $\partial _{k}A_{22}\in
L^{2}(\Omega )$) we deduce 
\begin{equation*}
\left\Vert \frac{\tau _{h}^{i}u_{0}-u_{0}}{h}\right\Vert _{L(\omega
_{1}^{\prime }\times \omega _{2})}\leq C^{\prime }+\left\Vert \left( \frac{%
\tau _{h}^{i}B(u_{0})-B(u_{0})}{h}\right) \right\Vert _{L^{2}(\omega
_{1}^{\prime }\times \omega _{2})}
\end{equation*}

Thanks to regularity of $B(u_{0})$ in the $X_{1}$ direction (assumption (\ref%
{43})) we get 
\begin{equation*}
\left\Vert \frac{\tau _{h}^{i}u_{0}-u_{0}}{h}\right\Vert _{L^{2}(\omega
_{1}^{\prime }\times \omega _{2})}\leq C^{\prime \prime },
\end{equation*}

where $C^{\prime \prime }$ is independent on $h$ , whence $\nabla
_{X_{1}}u_{0}\in L^{2}(\Omega )$ and the proof is finished.
\end{proof}

Now, we give a result on the rate of convergence

\begin{proposition}
Under assumptions of \textbf{Theorem 3} and \ (\ref{43}), (\ref{44}), for  $%
\beta >\max (K,\beta _{0})$ (where $\beta _{0}>M\left\vert \Omega
\right\vert ^{\frac{1}{2}-\frac{1}{r}}$ (fixed), and $K$  is the Lipschitz
constant of $B$ associated with the bounded set $\left\{ \left\Vert
u\right\Vert _{L^{2}}\leq \frac{M\left\vert \Omega \right\vert ^{\frac{1}{2}-%
\frac{1}{r}}}{\beta _{0}-M\left\vert \Omega \right\vert ^{\frac{1}{2}-\frac{1%
}{r}}}\right\} $), we have $u_{\epsilon }\rightarrow u_{0}$ in $W$ \ and 
\begin{equation*}
\left\Vert u_{\epsilon }-u_{0}\right\Vert _{L^{2}(\omega _{1}^{\prime
}\times \omega _{2})};\text{ }\left\Vert \nabla _{X_{1}}(u_{\epsilon
}-u_{0})\right\Vert _{L^{2}(\omega _{1}^{\prime }\times \omega _{2})}\leq
C^{\prime }\epsilon 
\end{equation*}

where $C\geq 0$ is independent of $\epsilon .$
\end{proposition}

\begin{proof}
To make calculus easier we suppose that $A_{12},A_{21}=0$, $A_{11},A_{22}=I$
. According to \textbf{Theorem 3} the set of solutions to (\ref{11}) is non
empty, and we show easily that (\ref{11}) has a unique solution (thanks to
assumption $\beta >\max (K,\beta _{0})$), consequently \textbf{Corollary 1}
implies $u_{\epsilon }\rightarrow u_{0}$ in $W$. 

From (\ref{6}) and (\ref{11}) we have%
\begin{equation*}
\epsilon ^{2}\dint\limits_{\Omega }\nabla _{X_{1}}u_{\epsilon }\nabla
_{X_{1}}\varphi dx+\dint\limits_{\Omega }\nabla _{X_{2}}(u_{\epsilon
}-u_{0})\nabla _{X_{2}}\varphi dx+\beta \dint\limits_{\Omega }(u_{\epsilon
}-u_{0})\varphi dx=\dint\limits_{\Omega }(B(u_{\epsilon })-B(u_{0}))\varphi
dx\text{ }
\end{equation*}%
Given $\omega _{1}^{\prime }\subset \subset $ $\omega _{1}^{^{\prime \prime
}}\subset \subset $ $\omega _{1}$, and let $\rho $ be a cut-off function
with $Supp(\rho )\subset \omega _{1}^{^{\prime \prime }}$ and $\rho =1$ on $%
\omega _{1}^{\prime }$(we can choose $0\leq \rho \leq 1$). We introduce the
test function used by M.Chipot and S.Guesmia in \cite{1}, $\varphi =$ $\rho
^{2}(u_{\epsilon }-u_{0})\in H_{0}^{1}(\Omega )$ ( thanks to the previous
proposition). Testing with $\varphi $ we obtain 
\begin{multline*}
\epsilon ^{2}\dint\limits_{\Omega }\nabla _{X_{1}}u_{\epsilon }\nabla
_{X_{1}}\rho ^{2}(u_{\epsilon }-u_{0})dx \\
+\dint\limits_{\Omega }\nabla _{X_{2}}(u_{\epsilon }-u_{0})\nabla
_{X_{2}}\rho ^{2}(u_{\epsilon }-u_{0})dx+\beta \dint\limits_{\Omega }\rho
^{2}(u_{\epsilon }-u_{0})^{2}dx \\
=\dint\limits_{\Omega }(B(u_{\epsilon })-B(u_{0}))\rho ^{2}(u_{\epsilon
}-u_{0})dx\text{ }
\end{multline*}

we deduce 
\begin{multline*}
\epsilon ^{2}\dint\limits_{\Omega }\left\vert \rho \nabla
_{X_{1}}(u_{\epsilon }-u_{0})\right\vert ^{2}dx+\dint\limits_{\Omega
}\left\vert \rho \nabla _{X_{2}}(u_{\epsilon }-u_{0})\right\vert ^{2}dx \\
+\beta \dint\limits_{\Omega }\rho ^{2}(u_{\epsilon }-u_{0})^{2}dx=-\epsilon
^{2}\dint\limits_{\Omega }\rho ^{2}\nabla _{X_{1}}u_{0}\nabla
_{X_{1}}(u_{\epsilon }-u_{0})dx-2\epsilon ^{2}\dint\limits_{\Omega
}(u_{\epsilon }-u_{0})\rho \nabla _{X_{1}}\rho \nabla _{X_{1}}u_{0}dx \\
-2\epsilon ^{2}\dint\limits_{\Omega }\rho (u_{\epsilon }-u_{0})\nabla
_{X_{1}}(u_{\epsilon }-u_{0})\nabla _{X_{1}}\rho dx+\dint\limits_{\Omega
}(B(u_{\epsilon })-B(u_{0}))\rho ^{2}(u_{\epsilon }-u_{0})dx
\end{multline*}

Using H\"{o}lder's inequality for the first three term in the right-hand
side, and assumptions (\ref{44}), (\ref{2}) and H\"{o}lder's inequality for
the last one, we obtain%
\begin{multline*}
\epsilon ^{2}\left\Vert \rho \nabla _{X_{1}}(u_{\epsilon }-u_{0})\right\Vert
_{L^{2}(\omega _{1}^{\prime \prime }\times \omega _{2})}^{2}+\left\Vert \rho
\nabla _{X_{2}}(u_{\epsilon }-u_{0})\right\Vert _{L^{2}(\omega _{1}^{\prime
\prime }\times \omega _{2})}^{2}+ \\
\beta \left\Vert \rho (u_{\epsilon }-u_{0})\right\Vert _{L^{2}(\omega
_{1}^{\prime \prime }\times \omega _{2})}^{2}\leq \epsilon ^{2}\left\Vert
\rho \nabla _{X_{1}}u_{0}\right\Vert _{L^{2}(\omega _{1}^{\prime \prime
}\times \omega _{2})}\left\Vert \rho \nabla _{X_{1}}(u_{\epsilon
}-u_{0})\right\Vert _{L^{2}(\omega _{1}^{\prime \prime }\times \omega _{2})}
\\
+2\epsilon ^{2}\left\Vert (u_{\epsilon }-u_{0})\nabla _{X_{1}}\rho
\right\Vert _{L^{2}(\omega _{1}^{\prime \prime }\times \omega
_{2})}\left\Vert \rho \nabla _{X_{1}}u_{0}\right\Vert _{L^{2}(\omega
_{1}^{\prime \prime }\times \omega _{2})} \\
+\epsilon ^{2}\left\Vert (u_{\epsilon }-u_{0})\nabla _{X_{1}}\rho
\right\Vert _{L^{2}(\omega _{1}^{\prime \prime }\times \omega
_{2})}\left\Vert \rho (u_{\epsilon }-u_{0})\right\Vert _{L^{2}(\omega
_{1}^{\prime \prime }\times \omega _{2})} \\
+K\left\Vert \rho (u_{\epsilon }-u_{0})\right\Vert _{L^{2}(\omega
_{1}^{\prime \prime }\times \omega _{2})}^{2},
\end{multline*}

(thanks to \textbf{Proposition 1, }we remark that $\left\Vert \rho
u_{\epsilon }\right\Vert _{L^{2}}$, $\left\Vert \rho u_{0}\right\Vert
_{L^{2}}\in \left\{ \left\Vert u\right\Vert _{L^{2}}\leq \frac{M\left\vert
\Omega \right\vert ^{\frac{1}{2}-\frac{1}{r}}}{\beta _{0}-M\left\vert \Omega
\right\vert ^{\frac{1}{2}-\frac{1}{r}}}\right\} $\textbf{)}$.$

Using Young's inequality for ,the first term in the right-hand side of the
previous inequality, and boundedness of $(u_{\epsilon })$ for the rest, we
deduce 
\begin{multline*}
\frac{\epsilon ^{2}}{2}\left\Vert \rho \nabla _{X_{1}}(u_{\epsilon
}-u_{0})\right\Vert _{L^{2}(\omega _{1}^{\prime \prime }\times \omega
_{2})}^{2}+\left\Vert \rho \nabla _{X_{2}}(u_{\epsilon }-u_{0})\right\Vert
_{L^{2}(\omega _{1}^{\prime \prime }\times \omega _{2})}^{2} \\
+(\beta -K)\left\Vert \rho (u_{\epsilon }-u_{0})\right\Vert _{L^{2}(\omega
_{1}^{\prime \prime }\times \omega _{2})}^{2}\leq C\epsilon ^{2}
\end{multline*}

whence 
\begin{equation*}
\left\Vert u_{\epsilon }-u_{0}\right\Vert _{L^{2}(\omega _{1}^{\prime
}\times \omega _{2})};\text{ }\left\Vert \nabla _{X_{2}}(u_{\epsilon
}-u_{0})\right\Vert _{L^{2}(\omega _{1}^{\prime }\times \omega _{2})}\leq
C^{\prime }\epsilon ,
\end{equation*}

where $C^{\prime }$ is independent of $\epsilon .$
\end{proof}

\subsection{Application to integro-differential problem}

In this section we provide some concrete examples. In \cite{7} M. Chipot and
S. Guesmia studied problem (\ref{6}) with the following integral operator 
\begin{equation}
B(u)=a\left( \dint\limits_{\omega _{1}}h(X_{1},X_{1}^{\prime
},X_{2})u(X_{1}^{\prime },X_{2})dX_{1}^{\prime }\right)  \label{39}
\end{equation}

To prove the convergence theorem\ the authors based their arguments on the
compacity of the operator $u\rightarrow \dint\limits_{\omega
_{1}}h(X_{1},X_{1}^{\prime },X_{2})u(X_{1}^{\prime },X_{2})dX_{1}^{\prime }$%
. Indeed, for a sequence $u_{n}\rightharpoonup u_{0}$ in $L^{2}(\Omega )$ we
have $\dint\limits_{\omega _{1}}h(X_{1},X_{1}^{\prime
},X_{2})u_{n}(X_{1}^{\prime },X_{2})dX_{1}^{\prime }$ $\rightarrow
\dint\limits_{\omega _{1}}h(X_{1},X_{1}^{\prime },X_{2})u_{0}(X_{1}^{\prime
},X_{2})dX_{1}^{\prime }$ in $L^{2}(\Omega )$ (by compacity) and we use the
continuity of $a$ and Lebesgue's theorem (under additional assumption on $a$%
) to get $a\left( \dint\limits_{\omega _{1}}h(X_{1},X_{1}^{\prime
},X_{2})u_{n}(X_{1}^{\prime },X_{2})dX_{1}^{\prime }\right) $ $\rightarrow
a\left( \dint\limits_{\omega _{1}}h(X_{1},X_{1}^{\prime
},X_{2})u_{0}(X_{1}^{\prime },X_{2})dX_{1}^{\prime }\right) $ in $%
L^{2}(\Omega ).$

We can give another operator\ based on the aforementioned one 
\begin{equation}
B(u)=\dint\limits_{\omega _{1}}h(X_{1},X_{1}^{\prime
},X_{2})a(u(X_{1}^{\prime },X_{2}))dX_{1}^{\prime },  \label{40}
\end{equation}

For $a:%
\mathbb{R}
\rightarrow 
\mathbb{R}
$ we note a Liptchitz function i.e there exists $K\geq 0$ such that%
\begin{equation}
\forall x,y\in 
\mathbb{R}
:\left\vert a(x)-a(y)\right\vert \leq K\left\vert x-y\right\vert  \label{37}
\end{equation}

In addition, we suppose that $a$ satisfies the growth condition 
\begin{equation}
\exists q\in \left[ 0,1\right[ ,\text{ }M\geq 0,\text{ }\forall x\in 
\mathbb{R}
:\left\vert a(x)\right\vert \leq M(1+\left\vert x\right\vert ^{q}),
\label{35}
\end{equation}

and we suppose that 
\begin{equation}
h\in L^{\infty }(\omega _{1}\times \Omega ),\text{ }\nabla _{X_{1}}h\in L^{%
\frac{2}{1-q}}(\omega _{1}\times \Omega )  \label{36}
\end{equation}

\begin{theorem}
Consider problem (\ref{6}) with $B$ given by (\ref{39}) or (\ref{40}).
Assume (\ref{1}), (\ref{3}), (\ref{34}), (\ref{37}), (\ref{35}) , (\ref{36})
and for $\beta $ suitably chosen, then we have the affirmations of \textbf{%
theorems 1 , 2 }and\textbf{\ 3 }of section 1 and those of \textbf{%
propositions 8, 9}
\end{theorem}

\begin{proof}
Take $B$ as in (\ref{40}) the proof of this theorem amounts to prove that
assumptions (\ref{2}), (\ref{4}), (\ref{5}), (\ref{43}) and (\ref{44}) hold.
(\ref{2}) follows directly from (\ref{37}) and (\ref{36}), Now assume (\ref%
{35}), (\ref{36}) then we can check easily that (\ref{4}) holds with $r=%
\frac{2}{q}.$ It remains to prove that (\ref{5}) holds. For every $u\in V$ (
we can also take $u\in L^{2}(\Omega )$)$,$ and $\varphi \in \mathcal{D}%
(\Omega )$ we have for $1\leq k\leq p$ 
\begin{multline*}
I(\varphi )=\left\vert \dint\limits_{\Omega }\left( \dint\limits_{\omega
_{1}}h(X_{1},X_{1}^{\prime },X_{2})a(u(X_{1}^{\prime },X_{2}))dX_{1}^{\prime
}\right) \partial _{k}\varphi (X_{1},X_{2})dX_{1}dX_{2}\right\vert  \\
=\left\vert \dint\limits_{\omega _{1}}\left( \dint\limits_{\Omega
}h(X_{1},X_{1}^{\prime },X_{2})\partial _{k}\varphi
(X_{1},X_{2})a(u(X_{1}^{\prime },X_{2}))dX_{1}dX_{2}\right) dX_{1}^{\prime
}\right\vert  \\
\leq \dint\limits_{\omega _{1}}\left\vert \left( \dint\limits_{\Omega
}h(X_{1},X_{1}^{\prime },X_{2})\partial _{k}\varphi
(X_{1},X_{2})a(u(X_{1}^{\prime },X_{2}))dX_{1}dX_{2}\right) \right\vert
dX_{1}^{\prime }
\end{multline*}

Since $\partial _{k}h\in L^{\frac{2r}{r-2}}(\omega _{1}\times \Omega )$ it
follows that for a.e $X_{1}^{\prime }\in \omega _{1}:$ $\partial _{k}\left[
a(u(X_{1}^{\prime },.))h(.,X_{1}^{\prime },.)\right] \in L^{\frac{2r}{r-2}%
}(\Omega )$, integrating by part we get 
\begin{eqnarray*}
I(\varphi ) &\leq &\dint\limits_{\omega _{1}}\left\vert \left(
\dint\limits_{\Omega }\partial _{k}h(X_{1},X_{1}^{\prime },X_{2})\varphi
(X_{1},X_{2})a(u(X_{1}^{\prime },X_{2}))dX_{1}dX_{2}\right) \right\vert
dX_{1}^{\prime } \\
&\leq &\left\Vert a(u)\right\Vert _{L^{r}}\left\vert \omega _{1}\right\vert
^{\frac{1}{2}}\left\Vert \partial _{k}h\right\Vert _{L^{\frac{2r}{r-2}%
}}\left\Vert \varphi \right\Vert _{L^{2}(\Omega )} \\
&\leq &M^{\prime }(1+\left\Vert u\right\Vert _{L^{2}})\left\Vert \varphi
\right\Vert _{L^{2}(\Omega )}
\end{eqnarray*}

And therefore $\partial _{k}B(u)\in L^{2}(\Omega )$, whence (\ref{43}) holds
and we have%
\begin{equation*}
\left\Vert \nabla _{X_{1}}B(u)\right\Vert _{L^{2}}\leq M^{\prime \prime
}(1+\left\Vert u\right\Vert _{L^{2}}),
\end{equation*}%
then for every $L^{2}-$bounded set $E\subset V$ we have%
\begin{equation}
\left\Vert \nabla _{X_{1}}B(u)\right\Vert _{L^{2}}\leq M^{\prime \prime
\prime },\text{ }u\in E.  \label{41}
\end{equation}%
\ Now, given a sequence $(U_{n})$ in $conv(B(E))$ which converges strongly
to some $U_{0}$ in $L^{2}(\Omega )$, by (\ref{41}) and the convexity of the
norm we show that $(\nabla _{X_{1}}U_{n})_{n}$ is bounded in $L^{2}(\Omega )$%
, hence one can extract a subsequence $(U_{n})$ such that $(\nabla
_{X_{1}}U_{n})$ converges weakly to some $c_{0}$ in $L^{2}(\Omega )$, thanks
to the continuity of derivation on $\mathcal{D}^{\prime }(\Omega )$ which
gives $c_{0}=\nabla _{X_{1}}U_{0}$ and therefore, $U_{0}\in V$ , whence (\ref%
{5}) follows. Finally, one can check easily that (\ref{44}) holds. Same
arguments when $B$ is given by (\ref{39})
\end{proof}

\subsection{\protect\bigskip A generalization}

Consider (\ref{39}) with%
\begin{equation}
h\in L^{\infty }(\Omega ),l\in L^{\infty }(\omega _{1}),\nabla _{X_{1}}l\in
L^{2}(\omega _{1}),  \label{45}
\end{equation}

the operator $u\rightarrow a\left( l(X_{1})\dint\limits_{\omega
_{1}}h(X_{1}^{\prime },X_{2})u(X_{1}^{\prime },X_{2})dX_{1}^{\prime }\right) 
$ belongs to a class of operators defined by%
\begin{equation}
B(u)=a\left( lP(u)\right) ,  \label{46}
\end{equation}

where $P:L^{2}(\Omega )\rightarrow L^{2}(\omega _{2})$ is a linear bounded
operator (an orthogonal projector for example). The method used by M. Chipot
and S. Guesmia is not applicable here, in fact the linear operator $P$ is
not necessarily compact, for $u_{n}\rightharpoonup u_{0}$ we only have $%
P(u_{n})\rightharpoonup P(u_{0})$ weakly and therefore every subsequence $%
(a\left( lP(u_{n})\right) )$ is not necessarily convergent in $L^{2}(\Omega )
$ strongly. However we have the following.

\begin{theorem}
Consider problem (\ref{6}) with $B$ given by (\ref{46}). Assume (\ref{1}), (%
\ref{3}), (\ref{34}), (\ref{37}), (\ref{35}) and (\ref{45}), then for $\beta 
$ suitably chosen, we have affirmations of \textbf{Theorems 1 , 2 }and%
\textbf{\ 3} of section 1 and moreover we have $u_{0}\in H^{1}(\Omega )$
\end{theorem}

\begin{proof}
The proof of this theorem amounts to prove that assumptions (\ref{2}), (\ref%
{4}), (\ref{5}) and (\ref{43}) hold. Since $P$ is Lipschitz then (\ref{2})
follows by (\ref{37}). We also can prove (\ref{4}) using (\ref{35}) with $r=%
\frac{2}{q}.$It remains to check that (\ref{5}), (\ref{43}) hold, for every $%
u\in V$ (we can take $u\in L^{2}(\Omega )$) we have $\nabla
_{X_{1}}a(lP(u))\in L^{2}(\Omega )$ and $\nabla _{X_{1}}a(lPu)=a^{\prime
}(lP(u))P(u)\nabla _{X_{1}}l.$ We can show easily that $\nabla
_{X_{1}}a(lP(E))$ is bounded for any $L^{2}-$bounded set $E\subset V$ and we
finish the proof as in \textbf{Theorem 8}.
\end{proof}


\begin{thebibliography}{9}
\bibitem{3} Cazenave T., An introduction to semilinear elliptic equations,
Editora do IM-UFRJ, Rio de Janeiro, 2006. ix+193 pp. ISBN: 85-87674-13-7.

\bibitem{1} M. Chipot, S. Guesmia, On the asymptotic behaviour of elliptic,
anisotropic singular perturbations problems,Com. Pur. App. Ana. 8 (1)
(2009), pp. 179-193.

\bibitem{7} M. Chipot, S. Guesmia, On a class of integro-differential
problems. Commun. Pure Appl. Anal. 9(5), 2010, 1249--1262.

\bibitem{2} M. Chipot, S.Guesmia, M. Sengouga. Singular perturbations of
some nonlinear problems. J. Math. Sci. 176 (6), 2011, 828-843.

\bibitem{6} J. L. Lions, \textquotedblleft Perturbations Singuli\`{e}res
dans les Probl\`{e}mes aux Limites et en Contr\^{o}le
Optimal,\textquotedblright\ Lecture Notes in Mathematics \# 323,
Springer-Verlag, 1973.

\bibitem{5} W. Rudin. functional analysis. McGraw-Hill Science.1991. ISBN :
0070542368.
\end{thebibliography}
\end{document}